\documentclass[reqno]{amsart}

\usepackage{tikz-cd}

\usepackage{fullpage,dsfont}

\usepackage{hyperref} 

\usepackage{parskip}
\makeatletter 
\def\thm@space@setup{%
 \thm@preskip=\parskip \thm@postskip=0pt
}
\def\th@remark{%
  \thm@headfont{\itshape}%
  \normalfont 
  \thm@preskip\parskip \thm@postskip=0pt
}
\makeatother

\usepackage[nobysame,alphabetic,initials,msc-links]{amsrefs}

\usepackage[final]{pdfpages}

\usepackage{multirow}

\usepackage{array}

\usepackage{stmaryrd}

\usepackage{tikz-cd}
\usetikzlibrary{knots}
\usepackage{spath3}

\usepackage{float}

\DefineSimpleKey{bib}{how}
\DefineSimpleKey{bib}{mrclass}
\DefineSimpleKey{bib}{mrnumber}
\DefineSimpleKey{bib}{fjournal}
\DefineSimpleKey{bib}{mrreviewer}

\renewcommand{\PrintDOI}[1]{%
  \href{http://dx.doi.org/#1}{{\tt DOI:#1}}%
}
\renewcommand{\eprint}[1]{#1}
\BibSpec{book}{%
    +{}  {\PrintPrimary}                {transition}
    +{.} { \PrintDate}                  {date}
    +{.} { \textit}                     {title}
    +{.} { }                            {part}
    +{:} { \textit}                     {subtitle}
    +{,} { \PrintEdition}               {edition}
    +{}  { \PrintEditorsB}              {editor}
    +{,} { \PrintTranslatorsC}          {translator}
    +{,} { \PrintContributions}         {contribution}
    +{,} { }                            {series}
    +{,} { \voltext}                    {volume}
    +{,} { }                            {publisher}
    +{,} { }                            {organization}
    +{,} { }                            {address}
    +{,} { }                            {status}
    +{,} { \PrintDOI}                   {doi}
    +{,} { \PrintISBNs}                 {isbn}
    +{}  { \parenthesize}               {language}
    +{}  { \PrintTranslation}           {translation}
    +{;} { \PrintReprint}               {reprint}
    +{.} { }                            {note}
    +{.} {}                             {transition}
    +{}  {\SentenceSpace \PrintReviews} {review}
}
\BibSpec{article}{%
    +{}  {\PrintAuthors}                {author}
    +{,} { \textit}                     {title}
    +{.} { }                            {part}
    +{:} { \textit}                     {subtitle}
    +{,} { \PrintContributions}         {contribution}
    +{.} { \PrintPartials}              {partial}
    +{,} { }                            {journal}
    +{}  { \textbf}                     {volume}
    +{}  { \PrintDatePV}                {date}
    +{,} { \issuetext}                  {number}
    +{,} { \eprintpages}                {pages}
    +{,} { }                            {status}
    +{,} { \PrintDOI}                   {doi}
    +{,} { \eprint}        {eprint}
    +{}  { \parenthesize}               {language}
    +{}  { \PrintTranslation}           {translation}
    +{;} { \PrintReprint}               {reprint}
    +{.} { }                            {note}
    +{.} {}                             {transition}
    +{}  {\SentenceSpace \PrintReviews} {review}
}
\BibSpec{collection.article}{%
    +{}  {\PrintAuthors}                {author}
    +{,} { \textit}                     {title}
    +{.} { }                            {part}
    +{:} { \textit}                     {subtitle}
    +{,} { \PrintContributions}         {contribution}
    +{,} { \PrintConference}            {conference}
    +{}  {\PrintBook}                   {book}
    +{,} { }                            {booktitle}
    +{,} { \PrintDateB}                 {date}
    +{,} { pp.~}                        {pages}
    +{,} { }                            {publisher}
    +{,} { }                            {organization}
    +{,} { }                            {address}
    +{,} { }                            {status}
    +{,} { \PrintDOI}                   {doi}
    +{,} { \eprint}        {eprint}
    +{}  { \parenthesize}               {language}
    +{}  { \PrintTranslation}           {translation}
    +{;} { \PrintReprint}               {reprint}
    +{.} { }                            {note}
    +{.} {}                             {transition}
    +{}  {\SentenceSpace \PrintReviews} {review}
}
\BibSpec{misc}{%
  +{}{\PrintAuthors}  {author}
  +{,}{ \textit}      {title}
  +{.}{ }             {how}
  +{}{ \parenthesize} {date}
  +{,} { available at \eprint}        {eprint}
  +{,}{ available at \url}{url}
  +{,}{ }             {note}
  +{.}{}              {transition}
}

\usepackage{amssymb, amsfonts, amsxtra, amsmath}
\usepackage{mathrsfs}
\usepackage{mathdots}
\usepackage{wasysym}
\usepackage[all]{xy}
\usepackage{bbm}
\usepackage{calc}
\usepackage{accents}
\usepackage{combelow}

\numberwithin{equation}{section}

\newtheorem{Theorem}{Theorem}[section]
\newtheorem*{Theorem*}{Theorem}
\newtheorem{Def}[Theorem]{Definition}
\newtheorem{Lem}[Theorem]{Lemma}
\newtheorem{Prop}[Theorem]{Proposition}
\newtheorem{Cor}[Theorem]{Corollary}

\newtheorem{Exa}[Theorem]{Example}

\newcommand\bp{\begin{proof}}
\newcommand\ep{\end{proof}}

\mathchardef\mhyph="2D

\DeclareMathOperator{\Ad}{\mathrm{Ad}}

\DeclareMathOperator{\End}{\mathrm{End}}

\DeclareMathOperator{\Hilb}{\mathrm{Hilb}}

\DeclareMathOperator{\id}{\mathrm{id}}

\DeclareMathOperator{\Rep}{\mathrm{Rep}}

\DeclareMathOperator{\Ker}{\mathrm{Ker}}

\DeclareMathOperator{\Lin}{\mathrm{Lin}}

\newcommand{\cop}{\mathrm{cop}}

\newcommand{\wt}{\mathrm{wt}}

\newcommand{\msC}{\mathscr{C}}
\newcommand{\msD}{\mathscr{D}}

\newcommand{\msF}{\mathscr{F}}
\newcommand{\msG}{\mathscr{G}}

\newcommand{\msH}{\mathscr{H}}
\newcommand{\msI}{\mathscr{I}}

\newcommand{\msM}{\mathscr{M}}

\newcommand{\msU}{\mathscr{U}}

\newcommand{\mfa}{\mathfrak{a}}

\newcommand{\mfb}{\mathfrak{b}}

\newcommand{\mfg}{\mathfrak{g}}

\newcommand{\mfh}{\mathfrak{h}}

\newcommand{\mfI}{\mathfrak{I}}
\newcommand{\mfk}{\mathfrak{k}}
\newcommand{\mfl}{\mathfrak{l}}

\newcommand{\mfn}{\mathfrak{n}}

\newcommand{\mfsl}{\mathfrak{sl}}
\newcommand{\mfso}{\mathfrak{so}}

\newcommand{\mfsu}{\mathfrak{su}}
\newcommand{\mft}{\mathfrak{t}}
\newcommand{\mfu}{\mathfrak{u}}

\newcommand{\mcD}{\mathcal{D}}

\newcommand{\Hsp}{\mathcal{H}}
\newcommand{\Gsp}{\mathcal{G}}
\newcommand{\mcI}{\mathcal{I}}

\newcommand{\mcO}{\mathcal{O}}

\newcommand{\mcU}{\mathcal{U}}

\newcommand{\C}{\mathbb{C}}

\newcommand{\R}{\mathbb{R}}

\newcommand{\Z}{\mathbb{Z}}

\newcommand{\opp}{\mathrm{op}}

\title{Quantisation of semisimple real Lie groups}
\author{K.\ De Commer}
\address{Vrije Universiteit Brussel}
\email{kenny.de.commer@vub.be}

\begin{document}
\maketitle

\begin{abstract}
We provide a novel construction of quantized universal enveloping $*$-algebras of real semisimple Lie algebras, based on Letzter's theory of quantum symmetric pairs. We show that these structures can be `integrated', leading to a quantization of the group C$^*$-algebra of an arbitrary semisimple algebraic real Lie group. 
\end{abstract}

\section*{Introduction}


The theory of quantum groups, as it was initiated through the works of M.\ Jimbo \cite{Jim85} and V.\ Drinfeld \cite{Dri87}, is by now an extensive framework with ramifications in many different areas of mathematics. The main object is a Hopf algebra $U_q(\mfg)$, depending on a parameter $q$ (either formal or scalar) and a semisimple (complex) Lie algebra $\mfg$, with $U_q(\mfg)$ deforming the classical universal enveloping algebra $U(\mfg)$ of $\mfg$. Suitably interpreted, one can say that $U(\mfg)$ arises in the limit as $q \rightarrow 1$. 

One can easily make sense of a \emph{compact} form of $U_q(\mfg)$. This entails providing $U_q(\mfg)$ with the structure of a Hopf \emph{$*$-algebra}, where $*$ is a particular anti-linear, anti-multiplicative and comultiplicative map on $U_q(\mfg)$. Classically, when $q=1$, this $*$-structure will restrict to an anti-multiplicative, anti-linear map $\mfg \rightarrow \mfg$ leading to the compact form 
\[
\mfu = \{X \in \mfg \mid X^* = -X\} \subseteq \mfg.
\]
The representation theory of $U_q(\mfu)$, meaning $*$-representations on (finite-dimensional) Hilbert spaces, can then be directly compared to the one of $\mfu$. For example, in both cases the irreducible representations can be naturally parametrized by (the same set of) highest weights. Through duality, it can also be connected directly to the operator algebraic framework of \emph{compact quantum groups} \cite{Wor87,LS91,DK94}.  


The quantisation of other real forms of $\mfg$ has known a much slower progress. Although there is a direct approach through the consideration of appropriate Hopf algebra automorphisms of $U_q(\mfg)$ \cite{Twi92}, the ensuing representation theory of these quantized real forms of $\mfg$, now necessarily on infinite-dimensional Hilbert spaces, has met with many stumbling blocks and analytical difficulties, see e.g.\ \cite{Wor91,Kor94,Wor00,KK03}. This has made it very difficult to build a satisfying theory in arbitrary rank.  

In this paper, we will consider a novel approach towards the quantisation problem for real semisimple Lie groups. This method requires more preparations on the algebraic side, but has as an immediate payoff that there are no longer any analytic issues at stake when considering the associated representation theory. The construction is built on two fundamental principles: 
\begin{itemize}
\item The maximal compact Lie subalgebra $\mfk$ of our real form $\mfl$ of $\mfg$ should be given by a \emph{symmetric pair coideal subalgebra} $U_q(\mfk)$ of $U_q(\mfu)$ \cite{Let99}. 
\item The associated quantized enveloping algebra $U_q(\mfl)$ should be obtained from $U_q(\mfk)$ through a generalisation of the \emph{Drinfeld double construction} to coideal subalgebras (see Section \ref{SecDrinfDoub}).   
\end{itemize}

This strategy was already explored in \cite{DCDz24}, where the specific case of $\mfl = \mfsl(2,\R)$ was considered, and in \cite{DCDz21}, where in more generality the concept of Drinfeld double coideals was developed in an operator algebraic setting. 

Here we take the opportunity to specifically set up the framework and formalism to consider quantizations of arbitrary semisimple  algebraic real Lie groups. Our scope is rather modest - we will simply gather the necessary ingredients from the literature to introduce and motivate our definition. Much of the remaining analysis (and hard work) will be left for future occasions.




This paper is organized as follows: in the \emph{first section}, we introduce the algebraic framework of \emph{Doi-Koppinen data and Doi-Koppinen modules}. In the \emph{second section}, we endow such Doi-Koppinen data with a unitary structure, so that the resulting representation theory on (pre-)Hilbert spaces can be considered. In the \emph{third section}, we show how Doi-Koppinen modules can be understood through a generalisation of the Drinfeld double construction. In the \emph{fourth section}, we then explain how Doi-Koppinen data can be obtained from a given Hopf $*$-algebra $U$ with a good representation theory, and a given left coideal $*$-subalgebra $I \subseteq U$. In the \emph{fifth section}, we explain how such inclusions $I\subseteq U$ arise naturally from the theory of  symmetric pair coideal subalgebras as developed by G. Letzter \cite{Let99}. In the \emph{sixth section}, we then apply the Drinfeld double construction to these latter coideals to arrive at the quantization of (the convolution algebra of) semisimple algebraic real Lie groups.

\emph{Acknowledgements}: Part of these results were presented at the XXXIX Workshop on Geometric Methods in Physics in Bia\l ystok, Poland, in 2022, at the Operator Algebra seminar at the Université de Caen Normandie in June 2023, and at the Conference `Quantum Groups and Noncommutative Geometry' in Prague in 2023. I thank all organizers and participating colleagues for their interest and for the opportunity to present this work. This research was funded by the FWO grant G032919N.

\emph{Notation}: if $V$ is a complex vector space, we denote by $\Lin_{\C}(V,\C)$ its linear dual. We occasionally use 
\[
V^{\circ} = \Lin_{\C}(V,\C)
\] 
as a short-hand. We write a generic $\C$-valued bilinear pairing as 
\[
\tau: V \times W \rightarrow \C,\qquad (v,w)\mapsto \tau(v,w). 
\]
We sometimes use this generic notation also when $V \subseteq \Lin_{\C}(W,\C)$ or $W \subseteq \Lin_{\C}(V,\C)$. Similarly, when $V_i$ is paired with $W_i$ for $i\in \{1,2\}$, we write by default  $\tau$ for the unique bilinear pairing between $V_1\otimes V_2$ and $W_1\otimes W_2$ such that
\[
\tau(v_1\otimes v_2,w_1\otimes w_2) = \tau(v_1,w_1)\tau(v_2,w_2),\qquad \forall v_i\in V_i,w_i\in W_i. 
\]


\section{Doi-Koppinen modules}

We recall some well-known constructions in the setting of Hopf algebras, see e.g.\ \cite{DNR01}.  

We work over the ground field $\C$, although for the moment a greater generality would be allowed. If $A = (A,m_A)$ is a unital algebra, we denote its unit by $1 = 1_A$, and if $(C,\Delta_C)$ is a (co-unital) coalgebra, we denote its counit by $\varepsilon = \varepsilon_C$. We then use the sumless Sweedler notation for the coproduct on $C$: 
\[
\Delta_C(c) = c_{(1)}\otimes c_{(2)},\qquad c\in C. 
\]
If $(A,\Delta_A)$ is a Hopf algebra, we denote its antipode by $S =S_A$. 

Modules over a unital algebra will always be assumed to be unital, and comodules over a counital coalgebra will be assumed counital. If $\delta_M: M \rightarrow M\otimes C$ is a right $C$-comodule and $\delta_N: N \rightarrow C\otimes N$ a left $C$-comodule, we accordingly write
\[
\delta_M(m) = m_{(0)}\otimes m_{(1)},\qquad \delta_N(n) = n_{(-1)}\otimes n_{(0)},\qquad m\in M,n\in N. 
\]
We can construct from the pair $(M,N)$ the \emph{cotensor product}, which is the vector space
\[
M \overset{C}{\square}N := \{z \in M \otimes N \mid (\delta_M\otimes \id)z = (\id\otimes \delta_N)z\}.
\]
We recall that a coalgebra $C$ is \emph{cosemisimple} if and only if there exist finite-dimensional vector spaces $(V_{\alpha})_{\alpha\in \mfI}$, indexed by some set $\mfI$, with
\begin{equation}\label{EqDecompCoalg}
C\cong \oplus_{\alpha \in \mfI} \End_{\C}(V_{\alpha})^{\circ}.
\end{equation}
Here we consider the direct sum coalgebra on the right, where each $\End_{\C}(V_{\alpha})^{\circ}$ is equipped with the coalgebra structure dual to the usual algebra structure on $\End_{\C}(V_{\alpha})$. If we choose a basis 
\[
\{e_i^{\alpha} \mid 1\leq i \leq \dim(V_{\alpha})\}
\] 
of each $V_{\alpha}$, and write $e_{ij}^{\alpha}$ for the associated matrix units and $\omega_{ij}^{\alpha}$ for the dual basis, the coproduct can be written explicitly as
\[
\Delta_C(\omega_{ij}^{\alpha}) = \sum_{k=1}^{\dim(V_{\alpha})} \omega_{ik}^{\alpha}\otimes \omega_{kj}^{\alpha},\qquad \alpha \in \mfI, 1\leq i,j\leq \dim(V_{\alpha}). 
\] 
The following structures were considered in \cite{Doi92,Kop95}. We follow the terminology of \cite{Sch00},  although the terminology of Doi-Hopf datum/Doi-Hopf module is also common \cite{CMS97}. 
\begin{Def}
A (left-right) \emph{Doi-Koppinen}  \emph{datum} consists of a triple $(A,B,C)$ with $A=(A,\Delta_A)$ a Hopf algebra, $B= (B,\delta_B)$ a right $A$-comodule algebra and $C = (C,\Delta_C)$ a left $A$-module coalgebra. 

A \emph{Doi-Koppinen} \emph{module} $(V,\pi_V,\delta_V)$ consists of a right $C$-comodule $V= (V,\delta_V)$, equipped with a left $B$-module structure $V = (V,\pi_V)$ such that the following compatibility condition holds: \begin{equation}\label{EqCompRelRel}
\delta_V(bv) = (\pi_V\otimes \id)(\delta_B(b))\delta_V(v),\qquad \forall b\in B,v\in V. 
\end{equation}
As usual, we drop the notation $\pi_V$ whenever it is clear by the context.
\end{Def}

We write the associated linear category of Doi-Koppinen modules as ${}_B\msM^C$.

An important class of Doi-Koppinen data can be constructed as follows. Let $(A,\Delta_A)$ be a Hopf algebra, and let $C$ be a quotient left $A$-module coalgebra of $A$ via a map
\[
\pi_C: A \twoheadrightarrow C.
\] 
If then $M, N$ are resp.\ right and left $C$-comodules, we consider the subspaces 
\[
M^C = \{m\in M\mid (\id\otimes \pi_C)\delta_M(m) = m\otimes \pi_C(1_A)\} \subseteq M,
\]
\[
{}^CN =  \{n\in N\mid (\pi_C\otimes \id)\delta_N(n) =\pi_C(1_A)\otimes n\} \subseteq N.  
\]
In particular, we obtain the unital \emph{subalgebra} 
\[
B := {}^CA.
\]
Here we think of $A$ as the function algebra on a quantum group, $C$ as the function algebra on a quantum subgroup, and $B$ as the algebra of functions on the associated homogeneous space of left cosets. The triple $(A,B,C)$ indeed forms a Doi-Koppinen datum, since one easily checks that 
\[
\Delta_A(B) \subseteq B\otimes A,
\]
i.e.\ $B\subseteq A$ is a \emph{right coideal} subalgebra. So, $\delta_B := (\Delta_A)_{\mid B}$ turns $B$ into a right $A$-comodule algebra.

\begin{Def}
We call a Doi-Koppinen datum of \emph{coideal type} $(A,{}^CA,C,\pi_C)$ any Doi-Koppinen datum arising in the above way from a Hopf algebra $A$ with left $A$-module coalgebra quotient $\pi_C: A\twoheadrightarrow C$. 
\end{Def}

In case of Doi-Koppinen data of coideal type, there is a tighter link between $B$ and $C$, leading one to expect more structure on ${}_B\msM^C$. This is indeed the case, as exemplified by the following theorem. The version we need is found in \cite{MS99}, but the general principle is already in \cite{Tak79}. We will need the category ${}_A^C\msM^C$ of $C$-bicomodules $(V,\delta_l,\delta_r)$ equipped with a compatible left $A$-module structure, the compatibility being that 
\[
(av)_{(0)}\otimes (av)_{(1)} = a_{(1)}v_{(0)}\otimes a_{(2)}v_{(1)},\quad (av)_{(-1)}\otimes (av)_{(0)} = a_{(1)}v_{(-1)}\otimes a_{(2)}v_{(0)},\qquad \forall a\in A,v\in V. 
\]
The latter category naturally carries a monoidal structure through the cotensor product:
\[
(M,N)\mapsto {}_{\bullet}^{\bullet}M\overset{C}{\square}\, {}_{\bullet}N^{\bullet}.
\] 
The bullets indicate where the relevant structure acts, e.g.\ the left $A$-module structure is the diagonal one. 
\begin{Theorem}\label{TheoTakeu}
Assume that $(A,B,C)$ is a Doi-Koppinen datum of coideal type, and assume that the antipode $S_A$ of $A$ is invertible. Assume moreover that $C$ is cosemisimple. Then the following hold:
\begin{enumerate}
\item $\Ker(\pi_C) = AB_+$, where $B_+ = \Ker((\varepsilon_A)_{\mid B})$.
\item $A$ is faithfully flat as a left and as a right $B$-module. 
\item There are quasi-inverse equivalences of categories 
\begin{equation}\label{EqEquivCatTak}
{}_B\msM^C \underset{F}{\overset{G}{\leftrightarrows}} {}_A^C\msM^C,\qquad
F(V) = {}^{\bullet}_{\bullet}A^{\bullet} \otimes_B V^{\bullet},\quad G(M) = {}_{\bullet}^CM^{\bullet}.
\end{equation}
The equivalence is implemented concretely by the isomorphisms
 \[
V \rightarrow G(F(V)),\quad v \mapsto 1_A \otimes v,\qquad F(G(M)) \rightarrow M,\quad a \otimes m \mapsto am.
\]
\end{enumerate} 
\end{Theorem}

A consequence of \eqref{EqEquivCatTak} is that ${}_B\msM^C$ inherits the monoidal structure of ${}^C_A\msM^C$, resulting in the following monoidal structure on ${}_B\msM^C$: 
\begin{equation}
V\boxtimes W := G(F(V)\overset{C}{\square} F(W)) = {}^C((A\otimes_B V)\overset{C}{\square}(A\otimes_B W))\cong {}_{\bullet}V \overset{C}{\square} ({}_{\bullet}A^{\bullet}\otimes_B W^{\bullet}). 
\end{equation}

There seems however to be no general way to implement the resulting tensor product more directly on ${}_B\msM^C$, say by endowing the usual vector space tensor product with an appropriate Doi-Koppinen module structure.

The following example is considered also in \cite{CMS97}. 
\begin{Exa}\label{ExaDiagonalCoid}
Let $H$ be a Hopf algebra with invertible antipode. Write $H^{\opp}$ for $H$ with the opposite product, but the original coproduct. Then on the tensor product Hopf algebra $A = H\otimes H^{\opp}$ we can consider 
\[
\pi_H: A \rightarrow H,\quad h \otimes k \mapsto hk,
\]
realizing $C := H$ as quotient left $A$-module coalgebra for the $A$-module structure
\[
(h\otimes k)\cdot c := hck,\qquad h,k,c\in A.
\]
It is easily checked that we can identify as algebras
\begin{equation}\label{EqRightCoideal}
H \cong B = {}^HA,\qquad h \mapsto h_{(2)}\otimes S_H^{-1}(h_{(1)}),
\end{equation}
the resulting right $A$-comodule structure on $H$ being
\[
\delta_B(h) = h_{(2)} \otimes (h_{(3)}\otimes S_H^{-1}(h_{(1)})). 
\]
We thus obtain a Doi-Koppinen datum $(A,B,C,\pi_H)  = (A,H,H,\pi_H)$ of coideal type. A Doi-Koppinen module is then nothing but a Yetter-Drinfeld module for $H$, i.e.\ a right $H$-comodule $(V,\delta)$ with left $H$-module structure interacting via 
\[
(hv)_{(0)}\otimes (hv)_{(1)} = h_{(2)}v_{(0)}\otimes h_{(3)}v_{(1)}S_H^{-1}(v_{(3)}). 
\]
Equivalence \eqref{EqEquivCatTak} is the well-known equivalence between the category of Yetter-Drinfeld modules and the category ${}^H_H\msM^H_H = {}_{H\otimes H^{\opp}}\!\!\!{}^H\msM^H$ of tetramodules (a cosemisimplicty assumption on $H$ or $A$ is in this case not needed).

We also have compatibility of the Doi-Koppinen and Yetter-Drinfeld tensor product for Yetter-Drinfeld modules under the above correspondence. Indeed, by the isomorphism
\[
{}^{\bullet}H \otimes H_{\bullet} \cong {}^{\bullet}A_{\bullet},\qquad h \otimes k \mapsto h k_{(2)}\otimes S_H^{-1}(k_{(1)}),\quad xy_{(1)}\otimes S_H(y_{(2)})\mapsfrom x\otimes y. 
\]
as left $H$-comodule and right $H$-module, it is clear that, for $V,W$ Yetter-Drinfeld modules, we can identify
\[
V\otimes W \cong {}_{\bullet}V \overset{H}{\square}(A^{\bullet}\otimes_B W^{\bullet}),\qquad v\otimes w \mapsto v_{(0)} \otimes ((v_{(1)}\otimes 1) \otimes w),
\]
the resulting tensor product on $V\otimes W$ indeed coinciding with the Yetter-Drinfeld one: 
\[
h(v\otimes w) = h_{(1)}v \otimes h_{(2)}w,\qquad \delta(v\otimes w) = v_{(0)}\otimes w_{(0)}\otimes v_{(1)}w_{(1)}. 
\]



\end{Exa}

\section{Unitary Doi-Koppinen modules}
\begin{Def}
A (unital) \emph{$*$-algebra} is a (unital) algebra $B$ equipped with an involutive, anti-linear, anti-multiplicative map 
\[
*: B \rightarrow B,\quad b \mapsto b^*,\qquad (ab)^* = b^*a^*.
\]
\end{Def}
\begin{Def}
Let $B$ be a $*$-algebra. If $\Hsp_0$ is a pre-Hilbert space, a \emph{$*$-representation} $\pi_0$ of $B$ on $\Hsp_0$ is a $B$-module structure $\pi_0$ on $\Hsp_0$ such that 
\[
\langle v,\pi_0(b)w\rangle = \langle \pi_0(b^*)v,w\rangle,\qquad \forall v,w\in \Hsp_0, \forall b\in B. 
\]

We say that the $*$-representation is \emph{bounded} if $\pi_0(b)$ is a bounded operator on $\Hsp_0$ for all $b\in B$. 
\end{Def}
If $(\Hsp_0,\pi_0)$ is a bounded $*$-representation, we can uniquely extend $\pi_0$ to a $*$-representation $\pi$ of $B$ on the Hilbert space completion $\Hsp$ of $\Hsp_0$. On the other hand, any Hilbert space $*$-representation is automatically bounded, by the uniform boundedness principle.

\begin{Def}
We say that a $*$-algebra $B$ is \emph{uniformly C$^*$-bounded} if any element of $B$ has a uniform bound for its norm with respect to any Hilbert space $*$-representation. 

We say that $B$ is \emph{strongly} uniformly C$^*$-bounded if moreover all of its pre-Hilbert space $*$-representations are bounded. 

We say that $B$ is \emph{C$^*$-faithful} if the $*$-representations of $B$ on Hilbert spaces separate the elements of $B$. 

We say that a $*$-representation $\pi$ of $B$ on a Hilbert space $\Hsp$ is \emph{non-degenerate} if $\pi(B)\Hsp$ is dense in $\Hsp$. 
\end{Def}
If $B$ is unital, then non-degeneracy of a $*$-representation is equivalent with $\pi$ being a unital $*$-representation, i.e.\ $\pi(1_B) = \id_{\Hsp}$. 



We now consider a dual notion \cite{ASH09,Chi18}. 
\begin{Def}
A \emph{$\dag$-structure} on a coalgebra $C$ is an involutive, anti-linear, anti-comultiplicative map 
\[
\dag: C \rightarrow C,\qquad c \mapsto c^{\dag},\qquad \Delta_C(c^{\dag}) = c_{(2)}^{\dag}\otimes c_{(1)}^{\dag}.
\]
\end{Def}
For example, if $\Hsp$ is a finite-dimensional Hilbert space, then $\End_{\C}(\Hsp)^{\circ}$ is a $\dag$-coalgebra by
\begin{equation}\label{EqDualForm}
\omega^{\dag}(x) := \overline{\omega(x^*)},\qquad x\in \End_{\C}(\Hsp),\omega \in \End_{\C}(\Hsp)^{\circ}.
\end{equation}
We will call a $\dag$-coalgebra of this form a \emph{basic (or simple) $C^{\dag}$-coalgebra}.

\begin{Def}
If $C$ is a $\dag$-coalgebra, a \emph{unitary right $C$-comodule} is any pre-Hilbert space $\Hsp_0$ equipped with a right $C$-comodule structure satisfying the following condition: 
\begin{equation}\label{EqLeftRightSame}
\langle v,w_{(0)}\rangle w_{(1)} = \langle v_{(0)},w\rangle v_{(1)}^{\dag},\qquad \forall v,w\in \Hsp_0.
\end{equation}
We then write this element as $U(v,w)\in C$, and refer to it as a \emph{matrix coefficient} of $\Hsp_0$.

We call $\Hsp_0$ \emph{locally complete} if, for any finite-dimensional subcoalgebra $D \subseteq C$, the space
\[
\Hsp_D = \{v\in \Hsp_0 \mid \delta(v)\in \Hsp_0\otimes D\}
\]
is complete, i.e.\ a \emph{Hilbert} space. 

We call $\Hsp_0$ \emph{locally finite} if moreover the $\Hsp_D$ above are all finite-dimensional. 
\end{Def}


In the following, we write $\Hilb^C$ for the category of locally complete unitary right $C$-comodules, with adjointable linear maps as morphisms. It is then in fact a $*$-category \cite[Definition 3.1]{Chi18}.  

\begin{Def}
We call $C$ a \emph{C$^\dag$-coalgebra} if it equals its set of matrix coefficients of unitary $C$-comodules. 
\end{Def}

It is elementary to show that a $\dag$-coalgebra $C$ is a C$^\dag$-coalgebra if and only if $C$ is a direct sum of basic $\dag$-coalgebras, i.e.\ there exist finite-dimensional Hilbert spaces $(\Hsp_{\alpha})_{\alpha\in \mfI}$ with
\begin{equation}\label{EqDecompCoalg}
C\cong \oplus_{\alpha \in \mfI} \End_{\C}(\Hsp_{\alpha})^{\circ}.
\end{equation} 

We write 
\[
\msI = \Lin_{\C}(C,\C)
\] 
for the \emph{full linear dual} of $C$. Then $\msI$ is a unital algebra under the convolution product
\[
\omega \chi := (\omega\otimes \chi)\Delta_C,
\] 
and, if $C$ is a $\dag$-coalgebra, $\msI$ becomes a $*$-algebra by means of \eqref{EqDualForm}. 

\begin{Def}
If $C$ is a $\dag$-coalgebra, we define its \emph{restricted dual} to be
\[
\mcI = \{x \in \msI\mid \msI x\msI \textrm{ is finite-dimensional}\}.
\]
\end{Def}

The subspace $\mcI$ is an ideal inside $\msI$, and hence typically a \emph{non-unital} $*$-subalgebra of $\msI$. In general, $\mcI$ can be quite small. If however $C$ is a C$^{\dag}$-coalgebra, the concrete identification \eqref{EqDecompCoalg} allows us to write 
\[
\msI\cong \prod_{\alpha \in \mfI} \End_{\C}(\Hsp_{\alpha}),\qquad \mcI \cong \oplus_{\alpha\in \mfI} \End_{\C}(\Hsp_{\alpha}),
\]
so in particular $\mcI$ is enough to determine $\msI$. 

Recall that we use $\tau$ as a generic notation for a bilinear pairing. Then when $C$ is a C$^{\dag}$-coalgebra, any unitary right $C$-comodule $(\Hsp_0,\delta)$ leads to a $*$-representation of $\msI$ on $\Hsp_0$ via 
\begin{equation}\label{EqDefIdMod}
\hat{\pi}_0(x)v = (\id\otimes \tau(-,x))\delta(v),\qquad x\in \msI,v\in \Hsp_0.
\end{equation}
The following proposition is elementary to prove from the fact that  $\mcI$ is isomorphic to  a direct sum of finite-dimensional matrix $*$-algebras.
\begin{Prop}\label{PropUnitComodEqui}
Let $(C,\Delta_C,\dag)$ be a C$^{\dag}$-coalgebra, and let $\mcI$ be the $*$-algebra introduced above. We then have that $\mcI$ is strongly uniformly C$^*$-bounded and C$^*$-faithful, and there is a one-to-one correspondence
\begin{equation}\label{EqOneToOneUnCom}
(\Hsp_0,\delta)\quad \leftrightarrow \quad (\Hsp,\hat{\pi})
\end{equation}
between locally complete unitary right $C$-comodules and non-degenerate $*$-representations of $\mcI$ on Hilbert spaces. 
\end{Prop}
Here $(\Hsp,\hat{\pi})$ is the completion of $(\Hsp_0,\hat{\pi}_0)$, and conversely given $\Hsp$ we define 
\begin{equation}\label{EqLocCompFin}
\Hsp_0 := \hat{\pi}(\mcI)\Hsp
\end{equation}
together with the $C$-comodule structure dual to the restriction $\hat{\pi}_0$ of $\hat{\pi}$ to operators $\Hsp_0 \rightarrow \Hsp_0$. 

Note that Proposition \ref{PropUnitComodEqui} presents one small subtlety: if we denote ${}_{\mcI} \Hilb$ the category of non-degenerate $\mcI$-representations on Hilbert spaces, we do \emph{not} obtain an equivalence of $*$-categories 
\begin{equation}\label{EqEquivCat}
\Hilb^C\rightarrow {}_{\mcI}\Hilb, 
\end{equation}
since the morphism spaces of the former are larger than the ones of the latter: norms of intertwiners need \emph{not} be uniformly bounded across the different components on the side of $\Hilb^C$!

We now recall the notion of CQG (= Compact Quantum Group) Hopf $*$-algebra \cite{DK94}. 

\begin{Def}
A \emph{Hopf $*$-algebra} is a Hopf algebra $A$ with a $*$-algebra structure preserving the coproduct: 
\[
\Delta_A(a^*) = a_{(1)}^*\otimes a_{(2)}^*.
\]
It is called a \emph{CQG Hopf $*$-algebra} if moreover there exists an invariant state $\Phi_A: A\rightarrow \C$, so 
\[
(\id\otimes \Phi_A)\Delta_A(a) = \Phi_A(a)1 = (\Phi_A\otimes \id)\Delta_A(a),\qquad \forall a\in A,
\]
\[
\Phi_A(1_A) = 1,\qquad \Phi_A(a^*a)\geq 0,\qquad \forall a\in A. 
\]
\end{Def}
We note that the invariant state on a CQG Hopf $*$-algebra is necessarily unique. 

We can endow a Hopf $*$-algebra $A$ with its canonical $\dag$-structure as a coalgebra, 
\[
a^{\dag} := S_A(a)^*,\qquad a\in A. 
\]
\begin{Prop}
The following conditions are equivalent for a Hopf $*$-algebra $A$:
\begin{itemize}
\item $A$ is a CQG Hopf $*$-algebra. 
\item The underlying $\dag$-coalgebra is a C$^{\dag}$-coalgebra. 
\item The underlying $*$-algebra is C$^*$-faithful and uniformly C$^*$-bounded.
\end{itemize}
The $*$-algebra $A$ is then automatically strongly uniformly C$^*$-bounded.
\end{Prop}

\begin{Def}
Let $A$ be a Hopf $*$-algebra. 

We call \emph{left $A$-module $\dag$-coalgebra} a left $A$-module coalgebra $C$ with a $\dag$-structure satisfying
\[
(ac)^{\dag} = a^{\dag}c^{\dag},\qquad \forall a\in A,c\in C.
\]

We call \emph{right $A$-comodule $*$-algebra} a right comodule algebra $B$ with a $*$-algebra structure satisfying
\[
\delta_B(b) = b_{(0)}^*\otimes b_{(1)}^*,\qquad \forall b\in B.
\]
\end{Def}

We can now introduce unitary Doi-Koppinen data and their unitary modules. 

\begin{Def}\label{DefUnitDrinfMod}
A \emph{unitary Doi-Koppinen datum} consists of a Doi-Koppinen datum $(A,B,C)$ with $A$ a Hopf $*$-algebra, $C$ a left $A$-module $\dag$-coalgebra, and $B$ a right $A$-comodule $*$-algebra. 

A \emph{unitary Doi-Koppinen module} $(\Hsp_0,\pi_0,\delta_0)$ for $(A,B,C)$ is a pre-Hilbert space $\Hsp_0$ with a Doi-Koppinen module structure for which its $C$-comodule is unitary and its $B$-module is a $*$-representation. 

A unitary Doi-Koppinen module is called 
\begin{itemize}
\item \emph{locally complete} if its underlying unitary $C$-comodule is locally complete,
\item \emph{admissible} if its underlying $C$-comodule is locally finite, and 
\item of \emph{Harish-Chandra type} if its underlying $C$-comodule is locally finite and finitely generated as a $B$-module.
\end{itemize}
We call two unitary Doi-Koppinen modules $\Hsp_0,\Gsp_0$ \emph{weakly equivalent} if there exists an invertible adjointable intertwiner $\Hsp_0\rightarrow \Gsp_0$, and \emph{strongly equivalent} if there exists a \emph{unitary} intertwiner $\Hsp_0\rightarrow \Gsp_0$. 

A locally complete unitary Doi-Koppinen module is called \emph{irreducible} if there are no non-trivial $B$-stable locally complete unitary $C$-subcomodules.
\end{Def}

A special class of unitary Doi-Koppinen data is obtained by looking at \emph{quotient} left $A$-module $\dag$-coalgebras $C$, i.e.\ the quotient map $\pi_C: A \twoheadrightarrow C$ is $\dag$-preserving. Now $B = {}^CA$ will be a unital $*$-subalgebra of $A$: If $b\in B$, then 
\[
\pi_C(b_{(1)}^*)\otimes b_{(2)}^* = \pi_C(S_A^{-1}(b_{(1)}))^{\dag} \otimes b_{(2)}^* = \pi_C(S_A^{-1}(b_{(2)})b_{(1)})^{\dag} \otimes b_{(3)}^* = \pi_C(1_A)\otimes b^*, 
\]
where in the second equality we used $b_{(1)}\in B$. It follows in particular that $B$ is a right $A$-comodule $*$-algebra. 

\begin{Def}\label{DefCQGCoid}
A unitary Doi-Koppinen datum $(A,B,C,\pi_C)$ is of \emph{coideal type} if it comes from a quotient $A$-module $\dag$-coalgebra $\pi_C: A \twoheadrightarrow C$ with $B = {}^CA$. If moreover $A$ is a CQG Hopf $*$-algebra, we say that $(A,B,C,\pi_C)$ is a CQG Doi-Koppinen datum of coideal type.
\end{Def}

\begin{Prop}\label{PropUnifBound}
If $(A,B,C,\pi_C)$ is a CQG Doi-Koppinen datum of coideal type, the following holds: 
\begin{itemize}
\item $B$ is strongly uniformly C$^*$-bounded and C$^*$-faithful.
\item $C$ is a C$^{\dag}$-coalgebra.
\end{itemize}
\end{Prop}
\begin{proof}
Let $b\in B$. Then $b$ lies in a finite-dimensional subcomodule $V \subseteq B$. Choosing an orthonormal basis $e_i \in V$ with respect to the inner product $\langle v,w\rangle = \Phi_A(v^*w)$, we find that $c =  \sum_i e_ie_i^*$ satisfies $\Delta(c) = c\otimes 1$, hence $c$ is a (positive) scalar. Then in any pre-Hilbert space $*$-representation $(\Hsp_0,\pi_0)$, we must have that each $\pi_0(e_i^*)$ is bounded, with norm $\leq c^{1/2}$. But then clearly also $\|\pi_0(b)\|\leq C_b$ for some $\pi_0$-independent constant $C_b$. 

The C$^*$-faithfulness of $B$ is immediate since $B \subseteq A$ and $A$ embeds (say) in its universal C$^*$-envelope. 

The fact that $C$ is a $C^{\dag}$-coalgebra is proven in \cite{Chi18}.
\end{proof} 


\section{Drinfeld doubles}\label{SecDrinfDoub}

In the following, we fix a CQG Doi-Koppinen datum $(A,B,C,\pi_C)$ of coideal type as in Definition \ref{DefCQGCoid}.

Consider the restricted and full dual 
\[
\mcI \subseteq \msI = \Lin_{\C}(C,\C).
\] 
If $\Hsp_0$ is a unitary Doi-Koppinen module, we obtain in particular a $*$-representation $\hat{\pi}_0$ of $\msI$ on $\Hsp_0$. It interacts with the $*$-representation $\pi_0$ of $B$ as follows: 
\begin{equation}\label{EqCorrRepToMake}
\hat{\pi}_0(x)\pi_0(b)v = \tau(b_{(1)}v_{(1)},x) \pi_0(b_{(0)})v_{(0)} = \pi_0(b_{(0)})\hat{\pi}_0(x\lhd b_{(1)})v,\qquad x\in \msI,b\in B,v\in V,
\end{equation}
where 
\[
\tau(c,x\lhd a) = \tau(ac,x),\qquad x\in \msI,a\in A,c\in C. 
\]
Note that $\lhd$ is simply the natural right $A$-module structure on the linear dual $\msI$ of $C$. 

This leads us to make the following definition, introduced in \cite{DCDz24}. It is a direct modification of \cite{Tak80,CMS97} to the case of Hopf $*$-algebras. 
\begin{Def}
We define $\msD(B,\msI)$ as the unital $*$-algebra generated by $B$ and $\msI$ with interchange relation 
\[
xb  = b_{(0)}(x\lhd b_{(1)}),\qquad x\in \msI,b\in B. 
\]
\end{Def}
It is easily seen that we then also have 
\[
bx =(x\lhd S_A^{-1}(b_{(1)})) b_{(0)},\qquad x\in \msI,b\in B,
\]
and that the following multiplication maps are bijective: 
\[
B\otimes \msI \rightarrow \msD(B,\msI),\quad b\otimes x\mapsto bx,\qquad 
\msI\otimes B \rightarrow \msD(B,\msI),\quad x\otimes b\mapsto xb. 
\]

Now using that 
\[
(xy)\lhd a = (x\lhd a_{(1)})(y\lhd a_{(2)}),\qquad x,y\in \msI,a\in A,
\]
we see that 
\[
z(x\lhd a)y = ((z\lhd S_A(a_{(1)}))x(y\lhd S_A^{-1}(a_{(3)})))\lhd a_{(2)},\qquad a\in A,x,y,z\in\msI. 
\]
It follows that $\mcI$ is stable under $\lhd A$, and so we obtain the following result. 
\begin{Prop}\label{PropDefReducedDual}
Define $\mcD(B,\mcI) := B\mcI \subseteq \msD(B,\msI)$. Then $\mcD(B,\mcI)$ is a $*$-subalgebra, and moreover the following multiplication maps are bijective: 
\[
B\otimes \mcI \rightarrow \mcD(B,\mcI),\quad b\otimes x\mapsto bx,\qquad 
\mcI\otimes B \rightarrow \mcD(B,\mcI),\quad x\otimes b\mapsto xb. 
\]
\end{Prop} 

The following theorem states that $\mcD(B,\mcI)$ governs the theory of Doi-Koppinen modules. 

\begin{Theorem}\label{TheoOtherOneOne}
If $\Hsp_0$ is a unitary Doi-Koppinen module, then the completion $\Hsp$ of $\Hsp_0$ carries a non-degenerate $*$-representation of $\mcD(B,\mcI)$ by 
\begin{equation}\label{EqDoiKop}
\pi_{\mcD}(bx) := \pi(b)\hat{\pi}(x),\qquad b\in B,x\in \mcI.
\end{equation}
Conversely, if $(\Hsp,\pi_{\mcD})$ is a non-degenerate $*$-representation of $\mcD(B,\mcI)$, then $\Hsp_0 := \pi_{\mcD}(\mcI) \Hsp$ carries a unique structure of unitary Doi-Koppinen module such that \eqref{EqDoiKop} holds. 

The above sets up a one-to-one correspondence between locally complete unitary Doi-Koppinen modules and non-degenerate $*$-representations of $\mcD(B,\mcI)$.
\end{Theorem}
\begin{proof}
We already know from \eqref{PropUnitComodEqui} that this theorem holds on the level of $\mcI$-representations. By Proposition \ref{PropUnifBound}, we also know that indeed a $*$-representation $\pi_0$ of $B$ on a pre-Hilbert space $\Hsp_0$ completes to a bounded $*$-representation on $\Hsp$. 

If then $\Hsp_0$ is a unitary Doi-Koppinen module, it follows directly from \eqref{EqCorrRepToMake} that \eqref{EqDoiKop} is a well-defined $*$-representation of $\mcD(B,\mcI)$ on $\Hsp$. It is clearly non-degenerate, since $\hat{\pi}_0(\mcI)\Hsp_0 = \Hsp_0$. 

Conversely, if $\Hsp$ is a a non-degenerate $*$-representation, it is only left to verify that $\Hsp_0 = \pi_{\mcD}(\mcI)\Hsp$ is stable under $\pi(B)$. But this is again an immediate consequence of \eqref{EqCorrRepToMake}.
\end{proof}

In general, we again as in \eqref{EqEquivCat} have that 
\begin{equation}\label{EqFunctorial}
{}_B\Hilb^C \rightarrow {}_{\mcD(B,\mcI)}\Hilb
\end{equation}
is not an equivalence of $*$-categories. However, if we restrict the left hand side to the category ${}_B\msH\msC^C$ of \emph{unitary Harish-Chandra Doi-Koppinen modules}, we do get an equivalence upon its image, which we then also refer to as (the $*$-category of) Harish-Chandra $*$-representations of $\mcD(B,\mcI)$. 

We similarly transport the notion of admissibility from Definition \ref{DefUnitDrinfMod} to the setting of $*$-representations of $\mcD(B,\mcI)$. As to the notion of irreducibility, we have the following lemma. Its proof is immediate by the fact that \eqref{EqFunctorial} respects and reflects isometric inclusions.

\begin{Lem}
A locally complete unitary Doi-Koppinen module $\Hsp_0$ is irreducible if and only if $\Hsp$ is irreducible as a $\mcD(B,\mcI)$-representation.
\end{Lem} 

\begin{Prop}\label{PropUnivEnve}
There exists a universal C$^*$-envelope $C^*(\mcD(B,\mcI))$ of $\mcD(B,\mcI)$, into which $\mcD(B,\mcI)$ embeds.  
\end{Prop} 
\begin{proof}
It is sufficient to prove that $\mcD(B,\mcI)$ is uniformally C$^*$-bounded and C$^*$-separated. 

The uniform C$^*$-boundedness of $\mcD(B,\mcI)$ is immediate, since both $B$ and $\mcI$ are (strongly) uniformly C$^*$-bounded. 

To see that $\mcD(B,\mcI)$ is C$^*$-separated, we check that $\mcD(B,\mcI)$ has a faithful $*$-representation on the pre-Hilbert space $A$ with inner product 
\[
\langle a,b\rangle = \Phi_A(a^*b),\qquad a,b\in A.
\]
Indeed, consider on $A$ the unitary Doi-Koppinen module structure 
\[
\pi_A(b)a  = ba,\qquad \delta_A(a) = (\id\otimes \pi_C)\Delta_A(a),\qquad a\in A,b\in B.
\]
To verify that the associated $*$-representation $\pi_{\mcD}$ of $\mcD(B,\mcI)$ is faithful, consider finitely many $b_i \in B$ and $\omega_i \in \mcI$ with 
\begin{equation}\label{EqGiveZeroAlways}
\sum_i b_i (\id\otimes \omega_i)\delta_A(a)=0,\qquad \forall a\in A,
\end{equation}
Then since $\Delta_A(A)(A\otimes 1) = A\otimes A$, multiplying \eqref{EqGiveZeroAlways} on the right with an arbitrary $a'\in A$ shows that 
\[
\sum_i b_i a \omega_i(c) = 0,\qquad \forall a\in A,c\in C,
\]
hence $\sum_i b_i \otimes \omega_i  =0$. This implies the faithfulness of $\pi_{\mcD}$. 
\end{proof}


\begin{Theorem}
Assume every irreducible $\mcD(B,\mcI)$-representation is admissible. Then $C^*(\mcD(B,\mcI))$ is type $I$. 
\end{Theorem}
\begin{proof}
It is sufficient to prove that every irreducible $*$-representation $\pi$ of $C^*(\mcD(B,\mcI))$ contains a compact operator (which may be taken as the definition of being type $I$). But by admissibility, any element $x\in \mcI$ with $\pi(x)\Hsp\neq \{0\}$ will give a non-zero finite rank operator $\pi(x) \in \pi(\mcD(B,\mcI))$. 
\end{proof}

We do not comment here on the monoidal structure that  ${}_{\mcD(B,\mcI)}\Hilb$ possesses. Roughly speaking, it is obtained by upgrading the equivalence of Theorem \ref{TheoTakeu} to the analytic level. See \cite[Section 2.2]{DCDz21} for more details.  


\section{A general construction method}

A common way to construct CQG Hopf $*$-algebras is as follows (on the abstract level, this is really just an instance of the Tannaka-Krein duality \cite{Wor88}). We first introduce the following notion. 

\begin{Def}\label{DefSComp}
Let $U$ be a Hopf $*$-algebra. A finite-dimensional unital $*$-representation $\hat{\pi}: U \rightarrow B(\Hsp_{\hat{\pi}})$ is called \emph{$S_U^2$-compatible} if there exists an invertible positive $T_{\hat{\pi}} \in B(\Hsp_{\hat{\pi}})$ such that 
\[
\hat{\pi}(S^2_U(x)) = T_{\hat{\pi}}\hat{\pi}(x)T_{\hat{\pi}}^{-1},\qquad \forall x\in U.
\] 
\end{Def}

\begin{Prop}\label{PropRepU}
Let $U$ be a Hopf $*$-algebra. Let $\msF = \{\hat{\pi}\}$ be a collection of $S_U^2$-compatible finite-dimensional unital $*$-representations of $U$. Then the unital $*$-algebra $A = A_{\msF} \subseteq \Lin_{\C}(U,\C)$ generated by the matrix coefficients 
\[
U_{\hat{\pi}}(\xi,\eta) \in \Lin(U,\C),\quad x \mapsto \langle \xi,\hat{\pi}(x)\eta\rangle
\]
is a CQG Hopf $*$-algebra for the convolution $*$-algebra structure
\[
\tau(ab,x) = \tau(a\otimes b,\Delta_U(x)),\qquad \tau(a^*,x) = \overline{\tau(a,S_U(x)^*)},\qquad a,b\in A,x\in U,
\]
and with the coproduct uniquely determined by 
\[
\tau(\Delta_A(U_{\hat{\pi}}(\xi,\eta)),x\otimes y) = \langle \xi,\hat{\pi}(xy)\eta\rangle,\qquad \xi,\eta\in \Hsp_{\hat{\pi}}.
\]
\end{Prop}
\begin{Def}\label{DefFtype}
Under the assumptions of Proposition \ref{PropRepU}, we call a finite-dimensional unital $*$-representation $\hat{\pi}$ of $U$ \emph{of $\msF$-type} if $U_{\hat{\pi}}(\xi,\eta) \in A$ for each $\xi,\eta\in \Hsp_{\hat{\pi}}$. More generally, we call a $*$-representation $\hat{\pi}_0$ of $U$ on a pre-Hilbert space $\Hsp_0$ \emph{of $\msF$-type} if 
\begin{itemize}
\item $\Hsp_{\xi} := \hat{\pi}_0(U)\xi$ is finite-dimensional for each $\xi\in \Hsp_0$, and
\item the restriction of $\hat{\pi}_0$ to each $\Hsp_{\xi}$ is of $\msF$-type.
\end{itemize}
We call $A$ the Hopf $*$-algebra of $\msF$-type matrix coefficients.
\end{Def}
There is then a one-to-one correspondence between $\msF$-type $*$-representations $(\Hsp_0,\hat{\pi}_0)$ of $U$ and unitary $A$-comodules $(\Hsp_0,\delta_0)$, the correspondence being that 
\begin{equation}\label{EqIndRep}
\hat{\pi}_0(x)\xi = (\id\otimes \tau(-,x))\delta_0(\xi),\qquad x\in U,\xi\in \Hsp_0. 
\end{equation}
Alternatively, if we denote $\msU$ the full dual of $A$, then we obtain a unital $*$-algebra homomorphism
\begin{equation}\label{EqUnivEmb}
U \rightarrow \msU, \quad x \mapsto \tau(-,x) \in \Lin_{\C}(A,\C),
\end{equation}
and it is easily seen that \eqref{EqIndRep} is just the factorisation through this $*$-homomorphism of the $\msU$-representation determined by \eqref{EqDefIdMod}.

The above set-up also gives a convenient way to construct unitary Doi-Koppinen data of coideal type \cite{MS99,Chi18}. 

\begin{Prop}\label{PropSetupCoideal}
Assume that $U$ is a Hopf $*$-algebra, and assume $A = A_{\msF}$ is a CQG Hopf $*$-algebra as above. Assume that $I$ is a left coideal $*$-subalgebra of $U$. Then the coimage $\pi_C:A \twoheadrightarrow C$ of  
\begin{equation}\label{EqQuotientCoalg}
A \rightarrow \Lin_{\C}(I,\C), \qquad a \mapsto \tau(a,-)_{\mid I}
\end{equation}
defines a left $A$-module quotient $\dag$-coalgebra $(C,\pi_C)$ of $A$. Moreover, if $B = {}^CA$, then 
\begin{equation}\label{EqHomSp}
B = \{b\in A\mid \forall x\in I: \tau(b_{(1)},x)b_{(2)} = \varepsilon_U(x)b\}. 
\end{equation}
\end{Prop}

We can also lift the construction of the Drinfeld double to this setting. 
\begin{Def}\label{EqDrinfeldDouble}
Assume the set-up of Proposition \ref{PropSetupCoideal}. We define $D(B,I)$ to be the universal $*$-algebra generated by copies of $B,I$ with interchange relations 
\[
xb= \tau(b_{(1)},x_{(-1)}) b_{(0)}x_{(0)},\qquad x\in I,b\in B.
\] 
\end{Def}
We then have again bijectivity of the multiplication maps 
\[
B \otimes I \cong D(B,I),\qquad I\otimes B\cong D(B,I). 
\]


\begin{Def}
A $*$-representation $\hat{\pi}_0$ of $I$ on a pre-Hilbert space $\Hsp_0$ is of $\msF$-type if 
\begin{itemize}
\item $\Hsp_{\xi} := \hat{\pi}_0(I)\xi$ is finite-dimensional for each $\xi\in \Hsp_0$, and
\item Each resulting restriction $\hat{\pi}_{0,\xi}: I \rightarrow B(\Hsp_{\xi})$ is isomorphic to an $I$-subrepresentation of an $\msF$-type $*$-representation $(\Gsp_0,\hat{\theta}_0)$ of $U$:
\[
\Hsp_{\xi} \subseteq \Gsp_0,\qquad \hat{\pi}_{0,\xi}(x) =  \hat{\theta}_0(x),\qquad \forall x\in I.  
\]
\end{itemize} 
We call a $*$-representation of $D(B,I)$ on a pre-Hilbert space of $\msF$-type if the underlying $I$-representation is of $\msF$-type. 
\end{Def}

\begin{Theorem}\label{TheoOneOneInt}
Assume $U$ is a Hopf $*$-algebra, and $\msF = \{\hat{\pi}\}$ a collection of $S_U^2$-compatible finite-dimensional unital $*$-representations of $U$. Assume $I$ is a left coideal $*$-subalgebra of $U$, and let $(A_{\msF},B,C,\pi_C)$ be the associated unitary Doi-Koppinen datum of coideal type as above.

Then there is a one-to-one correspondence between 
\begin{itemize}
\item $\msF$-type $*$-representations of $D(B,I)$, and 
\item unitary Doi-Koppinen modules for $(A_{\msF},B,C,\pi_C)$. 
\end{itemize}
\end{Theorem} 
\begin{proof}
Let us first note that there is a one-to-one correspondence between $\msF$-type $*$-representations $(\Hsp_0,\hat{\pi}_0)$ of $I$ and unitary right $C$-comodules $(\Hsp_0,\delta_0)$. Indeed, starting from the latter, we obtain by definition of $C$ a well-defined $*$-representation of $I$ through 
\begin{equation}\label{EqRelWeWant}
\hat{\pi}_0(x)\xi = (\id\otimes \tau(-,x))\delta_0(\xi),\qquad x\in I,\xi\in \Hsp_0.
\end{equation}
Then clearly each $\Hsp_{\xi} = \hat{\pi}_0(I)\xi$ is finite-dimensional.

To see that $\hat{\pi}_0$ is of $\msF$-type, it is enough to consider the case where $(\Hsp_0,\delta_0)$ is finite-dimensional and irreducible. Then by virtue of $C$ being a quotient coalgebra of $A$, we can find a finite-dimensional $\msF$-type $*$-representation $\hat{\rho}$ of $U$ such that 
\[
\{(\xi^* \otimes \id)\delta_0(\eta)\mid \xi,\eta\in V\} \subseteq \{\pi_C(U_{\hat{\rho}}(\xi,\eta))\mid \xi,\eta\in \Hsp_{\hat{\rho}}\}.
\]
But this means that $\hat{\pi}_0$ factors throught $\hat{\rho}(I)$, which by irreducibility of $\hat{\pi}_0$ is the same as $\hat{\pi}_0$ being a $*$-subrepresentation of $\hat{\rho}_{\mid I}$. 

Conversely, we note that if $\hat{\pi}_0$ is an $\msF$-type $*$-representation of $I$, we get by definition that there exists for each $\xi \in \Hsp_0$ a unitary right $C$-comodule structure $\delta_{\xi}: \Hsp_{\xi} \rightarrow \Hsp_{\xi}\otimes C$ such that 
\[
\hat{\pi}_0(x)\hat{\pi}_0(y)\xi = (\id \otimes \tau(-,x))\delta_{\xi}(\hat{\pi}_0(y)\xi),\qquad \forall x,y \in I.  
\]
However, by definition of $C$ this comodule structure is then uniquely determined by this condition. It is then straightforward to conclude that 
\[
\delta_0: \Hsp_0 \rightarrow \Hsp_0 \otimes C,\quad \xi\mapsto \delta_{\xi}(\xi)
\]
is a well-defined unitary $C$-comodule structure, related to $\hat{\pi}_0$ via \eqref{EqRelWeWant}.

The theorem is now easily concluded by noticing that if $(\Hsp_0,\hat{\pi}_0)$ is an $\msF$-type $*$-representation of $I$ with associated $C$-comodule $\delta_0: \Hsp_0\rightarrow \Hsp_0\otimes C$, and $\pi_0: B \rightarrow \End(\Hsp_0)$ is a unital $*$-representation, then
\begin{multline*}
\hat{\pi}_0(x)\pi_0(b) \xi = \tau(b_{(1)},x_{(-1)}) \pi_0(b_{(0)})\hat{\pi}_0(x_{(0)})\xi \textrm{ for all } x\in I,b\in B,\xi\in \Hsp_0\\
\iff \quad (b\xi)_{(0)}\otimes (b\xi)_{(1)} = b_{(0)}\xi_{(0)} \otimes b_{(1)}\xi_{(1)} \textrm{ for all }b\in B,\xi\in \Hsp_0.
\end{multline*}
\end{proof} 

Let us now say that a $D(B,I)$-representation of $\msF$-type is \emph{locally complete} if its associated unitary Doi-Koppinen module is locally complete. Then by combining Theorem \ref{TheoOneOneInt} with Theorem \ref{TheoOtherOneOne}, we obtain the following corollary.

\begin{Cor}
There is a one-to-one correspondence between locally complete $\msF$-type $*$-representations of $D(B,I)$, and non-degenerate $*$-representations of $\mcD(B,\mcI)$.
\end{Cor}

Similarly, let us say that an $\msF$-type $*$-representation of $D(B,I)$ is \emph{admissible} if each of its irreducible $I$-subrepresentations has finite multiplicity. Then we have the following compatibility between a priori different notions of irreducibility.

\begin{Prop}
Assume that $(\Hsp_0,\theta_0)$ is an admissible $\msF$-type $D(B,I)$-representation. Then the associated $\mcD(B,\mcI)$-representation $(\Hsp,\theta)$ is irreducible if and only if $(\Hsp_0,\theta_0)$ is irreducible as a $D(B,I)$-module.  
\end{Prop} 
\begin{proof}
This follows from the fact that, by admissibility, any $D(B,I)$-submodule of $\Hsp_0$ is automatically locally complete, leading to a one-to-one correspondence between closed $\mcD(B,\mcI)$-stable subspaces $V \subseteq \Hsp$ and $D(B,I)$-stable subspaces $V_0 \subseteq \Hsp_0$ through
\[
V \mapsto V_0 := V\cap \Hsp_0. 
\]
\end{proof}

\begin{Exa}\label{ExaComplexification}
We can fit Example \ref{ExaDiagonalCoid} into the setting of this section as follows. 

Let $U$ be a Hopf $*$-algebra, and let $\msF = \{\hat{\pi}\}$ be a collection of $S_U^2$-compatible finite-dimensional unital $*$-representations of $U$. Let $H = H_{\msF}$ be the associated CQG Hopf $*$-algebra of matrix coefficients. 

If $U^{\cop}$ is $U$ with the opposite coproduct, its antipode is given by $S_{U^{\cop}}= S_U^{-1}$. Hence $\msF := \{\hat{\pi}\}$ is still a collection of $S_{U^{\cop}}$-compatible  $*$-representations. The associated CQG Hopf $*$-algebra of matrix coefficients is given by the same vector space $H \subseteq \Lin_{\C}(U,\C)$, but now endowed with the opposite product and the new $*$-structure 
\[
h^{\star} := S_H^2(h)^*,\qquad h\in H.
\]
We denote this CQG Hopf $*$-algebra as $H^{\opp}$. 

Consider now the tensor product Hopf $*$-algebra $T = U\otimes U^{\cop}$, together with its family of $S_T^2$-compatible $*$-representations $\msG := \{\hat{\pi}\otimes \hat{\theta}\mid \hat{\pi},\hat{\theta}\in \msF\}$. Then the associated CQG Hopf $*$-algebra of matrix coefficients is the tensor product Hopf $*$-algebra $A := H\otimes H^{\opp}$, together with its natural pairing with $T$. 

If we now consider 
\[
\Delta: U \rightarrow T,\quad x \mapsto x_{(1)}\otimes x_{(2)},
\]
it is easily seen that $I := \Delta(U)$ is a left coideal $*$-subalgebra of $T$. The associated quotient $A$-module C$^{\dag}$-coalgebra is still given through the map 
\[
\pi_H: A \rightarrow H,\quad h \otimes k \mapsto hk,\qquad h,k\in H,
\]
where $H$ is endowed with its natural C$^\dag$-structure
\[
h^{\dag} = S_H(h)^*,\qquad h\in H.
\]
Through \eqref{EqRightCoideal}, we can then again realize $H$ as a right coideal subalgebra of $A$, now with compatible $*$-structure. 

If we compute the associated Drinfeld double coideal $D(H,U)$  as in Definition \ref{EqDrinfeldDouble}, we see that it is generated by the $*$-algebras $H,U$ with commutation relations
\[
xh = \tau(h_{(3)}\otimes S_H^{-1}(h_{(1)}),x_{(1)}\otimes x_{(3)})h_{(2)}x_{(2)} = \tau(S_H^{-1}(h_{(1)}),x_{(3)})h_{(2)}x_{(2)}\tau(h_{(3)},x_{(1)}),\qquad x\in U,h\in H. 
\]
So, $D(H,U)$ coincides with the usual Drinfeld double between the paired Hopf $*$-algebras $H,U$.
\end{Exa}



\section{Quantization of symmetric pairs}


Let $\mfg$ be a complex semisimple Lie algebra. By the fundamental work of \cite{Jim85,Dri87}, the universal enveloping algebra $U(\mfg)$ of $\mfg$ can be quantized, leading to a Hopf algebra $U_q(\mfg)$ depending on a parameter $q$. For our purposes, we already fix the condition that 
\[
q \,\textrm{ real with }0<q<1,
\]
as this will be important later on when considering associated $*$-structures. (The condition $1<q$ would also be allowed, but can be reduced to the case $q<1$ by symmetry.)

The precise form of $U_q(\mfg)$ that we will consider is as follows: we fix a Cartan subalgebra and Borel subalgebra 
\[
\mfh \subseteq \mfb \subseteq \mfg,
\]
and we let $Q$ be the associated root lattice with 
\begin{itemize}
\item associated root system $\Delta \subseteq Q$,
\item associated positive roots $\Delta^+ \subseteq \Delta$, and 
\item associated positive simple roots $I = \{\alpha_1,\ldots,\alpha_{\ell}\} \subseteq \Delta^+$.
\end{itemize}
We denote the associated weight lattice by $P \supseteq Q$, and we fix a positive-definite form $(-,-)$ on $Q \otimes_{\Z}\R$ which is invariant under the Weyl group $W$ and such that short roots $\alpha$ satisfy $(\alpha,\alpha)=2$. We then write $d_r = (\alpha_r,\alpha_r)/2$, and write $\alpha^{\vee} = 2\alpha/(\alpha,\alpha)$ for the associated coroots. We let $A = (a_{rs})_{rs}$ be the associated Cartan matrix under the convention 
\[
a_{rs} = (\alpha_r^{\vee},\alpha_s),\qquad 1\leq r,s\leq \ell.
\]
We also write 
\[
q_r = q^{d_r},\quad [n]_{q_r} = \frac{q_r^n-q_r^{-n}}{q_r-q_r^{-1}},\quad [n]_{q_r}! = [1]_{q_r}\ldots[n]_{q_r},\quad \binom{m}{n}_{q_r} = \frac{[m]_{q_r}!}{[n]_{q_r}![m-n]_{q_r}!}.  
\]

The precise conventions that we follow are then:
\begin{Def}
We define $U_q(\mfg)$ as the universal algebra generated by elements $K_{\omega}$ for $\omega\in P$, as well as elements $E_r,F_r$ for $1\leq r\leq \ell$, satisfying:
\begin{itemize}
\item $K_{\omega}K_{\chi} = K_{\omega+\chi}$ and $K_0 = 1$ for $\omega,\chi \in P$,
\item $K_{\omega}E_r = q^{(\omega,\alpha_r)}E_rK_{\omega}$ and $K_{\omega}F_r = q^{-(\omega,\alpha_r)}K_{\omega}F_r$ for $\omega\in P$ and $1\leq r\leq \ell$.
\item $E_rF_s -F_sE_r = \delta_{rs}\frac{K_{\alpha_r}-K_{\alpha_r}^{-1}}{q_r-q_r^{-1}}$ for $1\leq r,s\leq \ell$, and 
\item The quantum Serre relations for all $1\leq r\neq s\leq \ell$: 
\[
\sum_{t=0}^{1-a_{rs}}(-1)^t \binom{1-a_{rs}}{t}_{q_r} E_r^{t}E_sE_r^{1-a_{rs}-t} = 0,\quad \sum_{t=0}^{1-a_{rs}}(-1)^t \binom{1-a_{rs}}{t}_{q_r} F_r^{t}F_sF_r^{1-a_{rs}-t} = 0.
\]
\end{itemize}
We endow it with the unique Hopf algebra structure such that 
\[
\Delta(K_{\omega}) =K_{\omega}\otimes K_{\omega},\quad \Delta(E_r) = E_r\otimes 1 + K_{\alpha_r}\otimes E_r,\quad \Delta(F_r) = F_r\otimes K_{\alpha_r}^{-1}+ 1\otimes F_r,
\]
with counit and antipode determined by 
\[
\varepsilon(K_{\omega}) = 1,\quad \varepsilon(E_r) = \varepsilon(F_r) = 0,
\]
\[
S(K_{\omega}) = K_{-\omega},\quad S(E_r)= -K_{\alpha_r}^{-1}E_r,\quad S(F_r) = -F_rK_{\alpha_r}.
\] 
\end{Def}
We write 
\begin{itemize}
\item $U_q(\mfh)$ for the algebra generated by the $K_{\omega}$, \item $U_q(\mfb)$ for the algebra generated by the $K_{\omega}$ and $E_r$, and 
\item $U_q(\mfb^-)$   for the algebra generated by the $K_{\omega}$ and $F_r$.
\end{itemize}
Then these subalgebras are again universal with respect to the relations above involving them, and define natural Hopf subalgebras of $U_q(\mfg)$.

We can turn $U_q(\mfg)$ into a Hopf $*$-algebra by putting
\[
K_{\omega}^* = K_{\omega},\quad E_r^* = F_rK_{\alpha_r},\quad F_r^* = K_{\alpha_r}^{-1}E_r.
\]
We then write this Hopf $*$-algebra as $U_q(\mfu)$. Indeed, in the classical limit one has that $*$ determines an anti-linear Lie algebra involution $*: \mfg \rightarrow \mfg$ such that the real Lie algebra
\[
\mfu := \{X \in \mfg \mid X^* = -X\}
\]
is a compact real form of $\mfg$. We then also write $\mft = \mfu \cap \mfh$. 

\begin{Def}
A finite-dimensional unital $*$-representation $\hat{\pi}: U_q(\mfu) \rightarrow B(\Hsp_{\hat{\pi}})$ is called \emph{type 1} if $\pi(K_{\omega})$ is a positive operator for all $\omega \in P$. 
\end{Def}

It is immediate that type $1$-representations are $S^2$-compatible, since $S^2 = \Ad(K_{-2\rho})$ with 
\[
\rho = \frac{1}{2}\sum_{\alpha \in \Delta^+} \alpha.
\]
Now if $\Hsp_{\hat{\pi}}$ is a type $1$-representation, there exists a joint eigenbasis for the $\hat{\pi}(K_{\omega})$. If $\xi$ is such an eigenvector, we can uniquely write 
\[
\hat{\pi}(K_{\omega})\xi = q^{(\chi,\omega)}\xi,\qquad \forall \omega\in P,
\]
for some $\chi \in \R\otimes_{\Z}Q$. We call $\chi = \wt(\xi)$ the \emph{weight} of $\xi$. 

\begin{Def}
Assume $Q \subseteq F \subseteq P$ is a lattice. We say that a finite-dimensional unital $*$-representation of $U_q(\mfu)$ is of \emph{$F$-type} if any of its weight vectors has weight in $F$. 
\end{Def}
For example, it can be shown that any type $1$-representation is of type $P$. In general, we write 
\[
A := \mcO_q(U_F)
\]
for the Hopf $*$-algebra of $F$-type matrix coefficients, see Definition \ref{DefFtype}. 

Note that this definition has a natural classical analogue, in which $\mcO(U_P)$ is the algebra of regular functions on the unique connected, simply connected compact Lie group $U_P = U_{sc}$ integrating $\mfu$, while $\mcO(U_Q) \subseteq \mcO(U_P)$ is the  one attached to the quotient $U_Q = U_{ad} = U_{sc}/Z(U_{sc})$, the \emph{adjoint} Lie group associated to $\mfu$. 

We use the same notations for the complexification $G_F$ of $U_F$, where $\mcO_q(G_F)$ is simply viewed as $\mcO_q(U_F)$ with the $*$-structure forgotten. 

Continuing our notation from \eqref{EqUnivEmb}, we will then also write the linear dual of $\mcO_q(U_F)$ as
\[
\msU_q^{F}(\mfu) = \Lin(\mcO_q(U_F),\C).
\]
For any such $F$, we obtain an embedding of $*$-algebras
\[
U_q(\mfu) \subseteq \msU_q^F(\mfu).
\]

Assume now that we are given a subset of the simple roots,
\[
X \subseteq I.
\] 
Then we can consider $\mfg_X \subseteq \mfg$ as the Lie algebra generated by the root vectors associated to $X$. We denote $\Delta_X \subseteq \Delta$ for the roots obtained from the root vectors of $\mfg$ inside $\mfg_X$. Then $\Delta_X$ provides a copy of the root system of $\mfg_X$. We denote $\Delta_X^+ = \Delta_X \cap \Delta^+$. We denote by $W_X \subseteq W$ the Weyl group generated by the simple roots in $X$, and by $w_X$ its longest word. Finally, we denote 
\[
\delta_X^{\vee} =\frac{1}{2}\sum_{\alpha \in \Delta_X^+} \alpha^{\vee}.
\] 
\begin{Def}
A \emph{Satake diagram} for $\mfg$ (with fixed root data) consists of 
\begin{itemize}
\item a subset $X \subseteq I$ and
\item an involution $\tau: I \rightarrow I, \alpha_r \mapsto \tau(\alpha_r) = \alpha_{\tau(r)}$, 
\end{itemize}
such that 
\begin{itemize}
\item $\tau$ preserves the bilinear form on $I \subseteq Q$,
\item $\tau$ preserves $X$, and coincides on it with the action of $-w_X$, and 
\item $(\alpha,\delta_X^{\vee}) \in \Z$ for all $\alpha \in I \setminus X$ with $\tau(\alpha) = \alpha$.
\end{itemize}
\end{Def}

Satake diagrams are encoded on top of a Dynkin diagram by indicating the nodes in $X$ as black dots, and by indicating which nodes in $I\setminus X$ get swapped under $\tau$. An example of such a Satake diagram with underyling Dynkin diagram of type $A$ is given by 
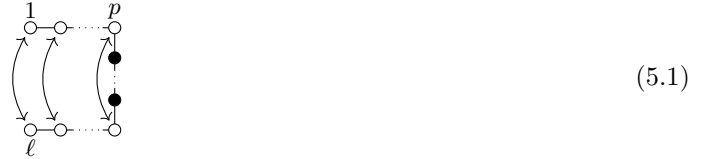
\begin{equation}\label{EqTypeA}
\begin{tikzpicture}[scale=.4,baseline={([yshift=-.5ex]current bounding box.center)}]
\node (v1) at (0,0) {};
\node (v2) at (1,0) {};
\node (v3) at (2.8,0) {};
\node (v4) at (2.8,-3.4) {};
\node (v5) at (1,-3.4) {};
\node (v6) at (0,-3.4) {};
\draw (0cm,0) circle (.2cm) node[above]{\small $1$};
\draw (0.2cm,0) -- +(0.6cm,0);
\draw (1cm,0) circle (.2cm);
\draw (1.2cm,0) -- +(0.2cm,0);
\draw[dotted] (1.4cm,0) --+ (1cm,0);
\draw (2.4cm,0) --+ (0.2cm,0);
\draw (2.8cm,0) circle (.2cm) node[above]{\small $p$};
\draw (2.8cm,-0.2cm) --+(0,-0.6cm);
\draw[fill = black] (2.8cm,-1cm) circle (.2cm); 
\draw (2.8cm,-1.2cm) --+ (0,-0.2cm);
\draw[dotted] (2.8cm,-1.4cm) --+ (0,-0.6cm);
\draw (2.8cm,-2cm) --+ (0,-0.2cm);
\draw[fill = black] (2.8cm,-2.4cm) circle (.2cm);
\draw (2.8cm,-2.6cm) --+ (0cm,-0.6cm);
\draw (2.8cm,-3.4cm) circle (.2cm);
\draw (0cm,-3.4cm) circle (.2cm) node[below]{\small $\ell$} ;
\draw (0.2cm,-3.4cm) -- +(0.6cm,0);
\draw (1cm,-3.4cm) circle (.2cm);
\draw (1.2cm,-3.4cm) -- +(0.2cm,0);
\draw[dotted] (1.4cm,-3.4cm) --+ (1cm,0);
\draw (2.4cm,-3.4cm) --+ (0.2cm,0);
\draw[<->]
(v1) edge[bend right] (v6);
\draw[<->]
(v2) edge[bend right] (v5);
\draw[<->]
(v3) edge[bend right] (v4);
\end{tikzpicture}
\end{equation}

Satake diagrams allow one to construct involutive Lie algebra automorphisms of $\mfu$. More precisely, if we linearly extend $\tau$ and consider 
\[
\Theta: Q \rightarrow Q,\quad \alpha \mapsto -w_X\tau(\alpha), 
\]
then there exists an involutive automorphism $\theta = \theta(X,\tau)$ of $\mfu$ whose complex linear extension to $\mfg$ permutes the root spaces as 
\[
\theta(\mfg_{\alpha}) = \mfg_{\Theta(\alpha)}. 
\]
Moreover, 
\begin{itemize}
\item any other such $\theta'$ is inner conjugate to $\theta$ by an element of $T = \exp(\mft) \subseteq G_{sc}$, so 
\[
\theta'= \Ad(t)\circ \theta\circ \Ad(t)^{-1},\qquad t\in T, 
\] 
and
\item any Lie algebra automorphism of $\mfu$ is inner conjugate to $\theta(X,\tau)$ for a unique $(X,\tau)$. 
\end{itemize}
We will refer to any such $\theta(X,\tau)$ as a \emph{Satake automorphism} of $\mfu$ (with respect to the fixed root data). 

\begin{Def}
If $\mfk$ is any Lie subalgebra of $\mfu$, we call $\mfk\subseteq \mfu$ a \emph{symmetric pair} if there exists a Lie algebra involution $\theta: \mfu \rightarrow \mfu$ such that
\[
\mfk = \mfu^{\theta} = \{X\in \mfu \mid \theta(X) = X\}.
\]
We call $\mfk \subseteq \mfu$ a \emph{symmetric pair of Satake type} if $\theta$ is a Satake automorphism.
\end{Def}

For example, the Satake diagram in \eqref{EqTypeA} encodes the inclusion 
\[
\mathfrak{s}(\mfu(p)\oplus \mfu(q))\subseteq \mfsu(p+q),\qquad p+q = \ell+1.
\]

As any Lie algebra involution of $\mfu$ will preserve the associated inner product coming from the Killing form on $\mfg$, it follows that $\mfk$ completely remembers $\theta$, with the $-1$-eigenspace $\mfu^{-\theta} = \mfk^{\perp}$. When quantizing, the primary focus will then be on the quantization of $\mfk \subseteq \mfu$, while the role of $\theta$ becomes less pronounced. 

The quantization of arbitrary symmetric pair Lie algebras was established by G. Letzter \cite{Let99}, with prior approaches for the classical types considered in e.g.\ \cite{NS95, Dij96, Nou96, BF97, DN98}. Subsequently, a generalisation to the Kac-Moody case was established in  \cite{Kol14} (see also the introduction of that paper for more details on the history of these concepts). 

The main new feature that arises in these constructions, is that the subsequent quantization $U_q(\mfk) \subseteq U_q(\mfu)$ is no longer a Hopf $*$-subalgebra, but only a \emph{left or right coideal $*$-subalgebra} - the choice of left vs. right is purely one of convention. Moreover, at least in the case where $\mfk$ has non-trivial center, there are some extra parameters that can be introduced in the quantization of $U_q(\mfk)$, corresponding to moving $\mfk$ away from its Satake position in a particular specified direction \cite{DCNTY23}. 

Before we move on to the precise construction of $U_q(\mfk)$, we make the following comments: 
\begin{itemize}
\item As mentioned, one has a version of $U_q(\mfk)$ as a left or as a right coideal $*$-subalgebra. One can canonically pass between the two choices using the \emph{unitary antipode} $R: U_q(\mfu) \rightarrow U_q(\mfu)$, which is a $*$-preserving anti-multiplicative, anti-comultiplicative involution determined by 
\[
R(K_{\omega}) = K_{-\omega},\quad R(E_r)= -q_r K_{\alpha_r}^{-1}E_r,\quad R(F_r) = -q_r^{-1} F_rK_{\alpha_r}.
\] 
The unitary antipode is simply a rescaling of the usual antipode $S$ as to become compatible with the $*$-structure. If then $I$ is a left coideal $*$-subalgebra, we obtain $J= R(I)$ as a right coideal $*$-subalgebra, and vice versa. 
\item In the original works on quantum symmetric pairs, the $*$-structure does not play any significant role, and compatibility with it was not considered, or not an essential requirement. However, it is crucial that our coideals are $*$-invariant to make the connection to the operator algebraic framework. We refer to \cite[Section 4]{DCM20} for a discussion on this.
\end{itemize}



Let us now introduce the particular form of $U_q(\mfk)$ that we will be interested in, following \cite[Section 4]{DCM20} (in particular, we do not consider the extra deformation parameters). Fix a Satake diagram $(X,\tau)$, and choose a function 
\[
z: I \rightarrow \{\pm1\},\quad z_r =1 \textrm{ when }(\alpha_r,\delta_X^{\vee}) \in \Z,\quad z_rz_{\tau(r)} = -1 \textrm{ when }(\alpha_r,\delta_X^{\vee})\notin \Z. 
\]
Such a function always exists, and its precise choice is not essential: different choices will create coideal $*$-subalgebras which can be transformed into each other under a Hopf $*$-algebra isomorphism of $U_q(\mfu)$ rescaling the generators $E_r,F_r$ by unimodular numbers. 

Recall further that the \emph{Lusztig braid operators} are particular elements $T_r \in \msU_q^P(\mfu)$ for $r\in I$, determined by 
\[
T_r\xi = \underset{-a+b-c = (\wt(\xi),\alpha_r^{\vee}}{\sum_{a,b,c\geq 0}}\frac{ (-1)^bq_r^{b-ac}}{[a]_{q_r}![b]_{q_r}![c]_{q_r}!}E_r^{a}F_r^{b}E_r^{c}\xi,\qquad \xi \in \Hsp_{\hat{\pi}} \textrm{ a type 1 $*$-representation}.
\]
They are invertible, and determine algebra automorphisms of $U_q(\mfu)$ through
\[
\Ad(T_r)(x) = T_rxT_r^{-1},\qquad x\in U_q(\mfu)\subseteq \msU_q^P(\mfu). 
\]
If then $X \subseteq I$ and $w_X = s_{r_1}\ldots s_{r_n}$ is the longest element in the Weyl group $W_X$, we can form the automorphism 
\[
\Ad(T_{w_X}) = \Ad(T_{r_1})\circ \ldots \circ \Ad(T_{r_n})
\]
of $U_q(\mfg)$,  which turns out to be independent of the choice of decomposition of $w_X$. We can now state: 
\begin{Def}
We define $U_q(\mfk) \subseteq U_q(\mfu)$ to be the unital subalgebra of $U_q(\mfu)$ generated by 
\begin{itemize}
\item the elements $E_r,F_r$ for $r\in X$,
\item the elements $K_{\omega}$ for $\omega \in P$ with $\omega = \Theta(\omega)$, and 
\item the elements\footnote{Note that in \cite{DCM20}, these elements were denoted as $C_r$.}
\[
B_r =E_r + q^{(\alpha_r^+,\alpha_r^+)}Y_rK_{\alpha_r},\qquad r\in I \setminus X,
\]
where 
\[
Y_r = -z_{\tau(r)}\Ad(T_{w_X})(F_{\tau(r)}),\qquad \alpha^+ = \frac{1}{2}(\alpha + \Theta(\alpha)). 
\]
\end{itemize}
\end{Def}
\begin{Prop}
The subalgebra $U_q(\mfk) \subseteq U_q(\mfu)$ is a left coideal $*$-subalgebra. 
\end{Prop} 
\begin{proof}
The coideal property can be proven along the lines of e.g.\  \cite[Proposition 5.2]{Kol14}, see also the discussion in \cite[Section 4]{DCM20}.  The $*$-invariance follows from \cite[Lemma 4.23]{DCM20}, see also \cite[Proposition 4.6]{BW18}. 
\end{proof} 

\begin{Exa}
Consider $\mfg = \mfsl(2,\C)$. Then we have associated to this the Satake diagram $(\emptyset,\id)$, giving the symmetric pair $\mfso(2) \subseteq \mfsu(2)$, with $\mfso(2)$ generated by $E-F$. The associated quantized enveloping algebra is given by 
\[
U_q(\mfso(2)) = \C[B],\qquad B = E- FK = E-E^*.
\]
The first detailed study of this case was carried out in \cite{Koo93}. 
\end{Exa}

\begin{Exa}\label{ExaDiagonal}
If we consider $\mfu\oplus \mfu$ with simple roots labeled by $I \sqcup I'$, then we can associate to it the Satake diagram $(\emptyset,\tau)$ with $\tau$ flipping the nodes: 
\[
\tau(r) = r',\qquad \tau(r') = r.
\]
Classically, the resulting inclusion is the diagonal inclusion 
\[
\mfu \subseteq \mfu \oplus \mfu. 
\]
In the quantized setting, we get that 
\[
U_q((\mfu\oplus \mfu)^{\theta}) \subseteq U_q(\mfu\oplus \mfu) = U_q(\mfu) \otimes U_q(\mfu)
\] 
is generated by all 
\[
K_{\omega}\otimes K_{\omega}^{-1},\qquad E_r\otimes 1 - q_rK_{\alpha_r}\otimes F_r,\quad 1\otimes E_r - q_r F_r \otimes K_{\alpha_r}.
\]
If we now apply to this the Hopf $*$-algebra isomorphism 
\[
\id\otimes \kappa: U_q(\mfu) \otimes U_q(\mfu) \rightarrow U_q(\mfu)\otimes U_q(\mfu)^{\cop},\qquad \kappa(K_{\omega}) = K_{\omega}^{-1},\quad \kappa(E_r) =- q_rF_r,\quad \kappa(F_r) = -q_r^{-1}E_r,
\]
we see that we are back in the situation of Example \eqref{ExaComplexification}.
\end{Exa}

\section{Quantization of semisimple algebraic real Lie groups}

Resume the setting of the previous section. Given an involution $\theta$ of $\mfu$, we can construct a new real Lie subalgebra of $\mfg$ by putting 
\[
\mfl = \mfl_{\theta} := \{X \in \mfg \mid X^* = -\theta(X)\}.
\]
Then $\mfl$ will be a semisimple real Lie algebra, and any semisimple real Lie algebra arises in this way. In fact, we see immediately that
\[
\mfk = \mfl \cap \mfu \subseteq \mfg,
\]
and this results into bijections 
\[
*\textrm{-stable real semisimple Lie subalgebras } \mfl \subseteq \mfg \quad \leftrightarrow \quad \textrm{symmetric pairs }\mfk \subseteq \mfu,
\]
known as \emph{Cartan duality}. In particular, real semisimple Lie algebras can also be encoded by Satake diagrams (although some care is needed in stating equivalences between Satake diagrams). The diagram in \eqref{EqTypeA} for example encodes the real Lie algebra $\mfsu(p,q)$. 

Assume now that we have chosen a lattice $Q \subseteq F \subseteq P$, and let $U_F \subseteq G_F$ be the associated Lie groups. If $\theta = \theta(X,\tau)$ is a Satake involution such that $\tau(F) = F$, then $\theta$ can be integrated to a (complex) Lie group involution 
\[
\theta: G_F \rightarrow G_F,\qquad \theta(U_F) = U_F. 
\] 
In particular, we can consider 
\[
L_F  = L_F^{\theta}:= \{g\in G_F \mid g^* = \theta(g)^{-1}\},
\]
which will be a Lie subgroup of $G_F$ with $\mfl$ as its Lie algebra. If $F= P$, then $L_P$ will be connected, but this need not be the case in general, e.g.\ one can consider (with respect to the properly chosen Cartan system) the inclusion $SO(m,n)\subseteq SO(m+n,\C)$ for $m,n\geq 1$, arising from the involution 
\[
\theta: x \mapsto \begin{pmatrix} I_m & 0 \\ 0 & -I_n\end{pmatrix} x \begin{pmatrix} I_m & 0 \\ 0 & -I_n\end{pmatrix}.
\]

Following now the notations of \eqref{EqQuotientCoalg} and \eqref{EqHomSp}, we have associated to $U_q(\mfk)$ the quotient left $\mcO_q(U_F)$-module coalgebra 
\[
\pi_{K_F}: \mcO_q(U_F) \rightarrow \mcO_q(K_F)
\]
arising as the coimage of the map 
\[
\pi_{K_F}: \mcO_q(U_F) \rightarrow \Lin_{\C}(U_q(\mfk),\C),\quad  U_{\hat{\pi}}(\xi,\eta) \mapsto (x\mapsto \langle \xi,\hat{\pi}(x)\eta\rangle),\quad \hat{\pi}\textrm{ of $F$-type},\xi,\eta\in \Hsp_{\hat{\pi}},x\in U_q(\mfk),
\]
and the associated right coideal $*$-subalgebra
\[
\mcO_q(K_F\backslash U_F) = \{U_{\pi}(\xi,\eta) \in \mcO_q(U_F)\mid \hat{\pi}\textrm{ of $F$-type},\xi,\eta\in \Hsp_{\hat{\pi}}, x\xi = \varepsilon(x)\xi\textrm{ for all }x\in U_q(\mfk)\}.
\]

Recall now Definition \ref{EqDrinfeldDouble} and the definitions given in Proposition \ref{PropDefReducedDual} and Proposition \ref{PropUnivEnve}.


\begin{Def}
We define the \emph{quantized enveloping algebra} of $\mfl$ to be
\begin{equation}\label{EqIwaInf}
U_q(\mfl) := D(\mcO_q(K_P\backslash U_P),U_q(\mfk)).
\end{equation}
We define the \emph{quantized convolution algebra} of $L_F$ to be
\begin{equation}\label{EqIwaInt}
\mcU_q(L_F) = \mcD(\mcO_q(K_F \backslash U_F), \mcU_q(K_F)).
\end{equation}
We define the \emph{universal group C$^*$-algebra} of $L_F$ to be 
\begin{equation}\label{EqIwaCstar}
C^*_q(L_F) = C^*(\mcU_q(L_F)).
\end{equation}
\end{Def}

This requires a word of explanation. Recall that if consider $\mfg$ as a real Lie algebra and put 
\[
\mfg = \mfu \oplus \mfa \oplus \mfn
\]
for the associated \emph{Iwasawa decomposition}, then $(\mfg,\mfu,\mfa\oplus \mfn)$ has the structure of a \emph{Manin triple} (see e.g.\ \cite{CP95}). We can then view the associated Iwasawa decomposition
\[
G_F = U_FA_FN_F
\]
of $G_F$ as a particular integrated version $(G_F,U_F,A_FN_F)$ of this Manin triple. By Drinfeld duality \cite{Dri87, ES98}, this means that we should interpret $\mcO_q(U_F)$ as a quantized enveloping algebra of $\mfa\oplus \mfn$, or more accurately as a quantized convolution $*$-algebra $\mcU_q(A_FN_F)$ of $A_FN_F$ (meaning a quantization of the convolution $*$-algebra of compact support functions on $A_FN_F$).

Now as $\mfl \subseteq \mfg$ was chosen to be in Satake form, it follows that the Iwasawa decomposition of $\mfl$ is obtained immediately from the one for $\mfg$: 
\[
\mfl = \mfk \oplus \mfa^0 \oplus \mfn^0,\quad \mfk = \mfu \cap \mfl,\quad \mfa^0 = \mfa \cap \mfl,\quad \mfn^0 = \mfn\cap \mfl.
\]
If we then view the associated integrated Iwasawa decomposition
\[
L_F = K_F A_F^0N_F^0,
\]
another instance of Drinfeld duality, this time for coisotropic Lie subalgebras \cite{Dri93,CG06}, allows us to view 
\[
\mcO_q(K_L\backslash G_L) = \mcU_q(A_{F}^0N_{F}^0).
\]
It is now clear that \eqref{EqIwaInf} and \eqref{EqIwaInt} may be seen more precisely as quantizations of $U(\mfl)$, resp.\ $\mcU(L_F)$, together with their Iwasawa decomposition. Here $\mcU(L_F)$ is to be seen as those functions in the convolution $*$-algebra of compact support functions on $L_F$ that generate a finite-dimensional subspace when translated from the left (or right) with $U_F$. 

The precise way in which $\mcU_q(L_F)$ gives rise to $\mcU(L_F)$ in the classical limit will not be dealt with here.


 


We believe the above setup to be the appropriate one for the quantisation of (linear, algebraic) semisimple real Lie groups. For example, when viewing the constructions of Definition \ref{DefUnitDrinfMod} for the triple 
\[
(\mcO_q(U_F),\mcO_q(K_F\backslash U_F),\mcO_q(K_F)),
\] 
one immediately gets the appropriate notions for the corresponding ones in the classical setting of semisimple real Lie groups. We also stress that, although $U_q(\mfl)$ or $\mcU_q(L_F)$ do not have any Hopf algebra structure themselves (for example $U_q(\mfl)$ is only a left coideal inside the Drinfeld double $\mcD(U_q(\mfu),\mcO_q(U_P))$), there is nevertheless a monoidal structure on the representation category 
\[
\Rep_q(L_F) := \Rep_*(C_q^*(L_F)),
\]  
as elaborated on in \cite[Section 2.2]{DCDz21}.

At the moment, there are still some natural extensions of the classical theory to be considered: 
\begin{itemize}
\item A theory of induction from quantum parabolic subalgebras needs to be developed. Here there are two questions to address: 
\begin{itemize}
\item What are the appropriate general quantum parabolic subalgebras to consider in the case of $U_q(\mfl)$?
\item How to correctly implement the associated induction of $*$-representations on the operator algebraic level? This question can already be addressed directly for the quantum parabolic subalgebra $\mcO_q(K_L\backslash U_L)$ (i.e. the construction of \emph{principal series} representations).
\end{itemize}
\item Through the formalism of  \cite{DCDz21}, a natural weight should become available on $C^*_q(L_F)$. It should then be feasible to obtain an associated Plancherel formula for $C^*_q(L_F)$. The case of $L_F = SL(2,\R)$ is currently being examined by the author and J. Dzokou Talla.
\item A classification, or at least a construction of large classes of irreducible representations of the $C^*_q(L)$ should be feasible through either classical techniques, or by constructions accessible only in the quantum realm. For $L_F = SL(2,\R)$, a full classification was achieved in \cite{DCDz24}.  
\end{itemize}

One question that should be feasible to answer quite directly, is whether any irreducible $*$-representation of $\mcU_q(\mfl)$ is admissible. Unfortunately, a resolution of this problem could not be obtained in time for the submission of this article.

Note that for the diagonal inclusion considered in Example \ref{ExaDiagonal}, we find through Example \ref{ExaComplexification} that the above considerations become the ordinary ones for the usual Drinfeld double of $U_q(\mfg)$ and $\mcO_q(U_P)$. In this case, the resulting $*$-algebra $U_q(\mfg_{\R})$ will quantize the complex Lie algebra $\mfg$ as a real Lie algebra, and \emph{will} be equipped with a coproduct (as a Drinfeld double of Hopf $*$-algebras). In this case, the above questions have been (mostly) fully answered through the works of \cite{Ar17,Ar18,VY20,VY23}.



\begin{thebibliography}{00}
\bibitem[ASH09]{ASH09} A. Abella, W.F. Santos, M. Haim, Compact coalgebras, compact quantum groups and the positive antipode, \emph{São Paulo J. Math. Sci.} \textbf{3} (2) (2009), 193--229.
\bibitem[Ar17]{Ar17} Y. Arano, Comparison of unitary duals of Drinfeld doubles and complex semisimple Lie groups, \emph{Comm. Math. Phys.} \textbf{351} (3) (2017), 1137--1147.
\bibitem[Ar18]{Ar18} Y. Arano, Unitary spherical representations of Drinfeld doubles, \emph{J. Reine Angew. Math.} \textbf{742} (2018), 157--186.
\bibitem[BF97]{BF97} W. Baldoni and P. M. Frajria, The quantum analog of a symmetric pair: A construction in type $(C_n, A_1 \times C_{n-1})$, \emph{Trans. Amer. Math. Soc.} \textbf{8} (1997), 3235--3276.
\bibitem[BW18]{BW18} H. Bao, W. Wang, Canonical bases arising from quantum symmetric pairs, \emph{Invent. Math.} \textbf{213} (3) (2018), 1099--1177.
\bibitem[CMS97]{CMS97} S. Caenepeel, G. Militaru and Z. Shenglin, Crossed modules and Doi-Hopf modules, \emph{Israel J. Math.} \textbf{100} (1997), 221--247.
\bibitem[CP95]{CP95} V. Chari, A. N. Pressley, A guide to quantum groups, \emph{Cambridge University Press} (1995).
\bibitem[Chi18]{Chi18} A. Chirvasitu, Relative Fourier transforms and expectations on coideal subalgebras, \emph{J. Algebra} \textbf{516} (2018), 271--297.
\bibitem[CG06]{CG06} N. Ciccoli and F. Gavarini, A quantum duality principle for coisotropic subgroups and Poisson quotients, \emph{Adv. Math.} \textbf{199} (2006), 104--135.
\bibitem[DNR01]{DNR01} S. D\u{a}sc\u{a}lescu, C. N\u{a}st\u{a}sescu, \cb{S}. Raianu, Hopf algebras: an introduction, \emph{Monographs and Textbooks in Pure and Applied Mathematics} \textbf{235}, Marcel Dekker (2001).
\bibitem[DCY13]{DCY13} K. De Commer and M. Yamashita, Tannaka-Kre\u{i}n duality for compact quantum homogeneous spaces. I. General Theory, \emph{Theory Appl. Categ.} \textbf{28} (31) (2013), 1099--1138.
\bibitem[DCM20]{DCM20} K. De Commer and M. Matassa, Quantum flag manifolds, quantum symmetric spaces and their associated universal K-matrices, \emph{Adv. Math.} \textbf{366} (2020), 107029.
\bibitem[DCDz21]{DCDz21} K. De Commer, J. R. Dzokou Talla, Invariant integrals on coideals and their Drinfeld doubles, \emph{preprint}, arXiv:2112.07476.
\bibitem[DCDz24]{DCDz24} K. De Commer and J.R. Dzokou Talla, Quantum $SL(2,\R)$ and its irreducible representations, \emph{J. Operator Theory} \textbf{91} (2) (2024), 101--128.
\bibitem[DCNTY23]{DCNTY23} K. De Commer, S. Neshveyev, L. Tuset and M. Yamashita, Comparison of quantizations of symmetric spaces: cyclotomic Knizhnik-Zamolodchikov equations and Letzter-Kolb coideals, \emph{Forum of Mathematics, Pi} \textbf{11} (14) (2023), doi:10.1017/fmp.2023.11.
\bibitem[DK94]{DK94} M.S. Dijkhuizen and T.H. Koornwinder, CQG algebras: a direct algebraic approach to compact quantum groups, \emph{Lett. Math. Phys.} \textbf{32} (1994), 315--330.
\bibitem[Dij96]{Dij96} M.S. Dijkhuizen, Some remarks on the construction of quantum symmetric spaces, In:  \emph{Representations of Lie Groups, Lie Algebras and Their Quantum Analogues, Acta Appl. Math.} \textbf{44} (1-2)  (1996), 59--80.
\bibitem[DN98]{DN98} M.S. Dijkhuizen and M. Noumi, A family of quantum projective spaces and related q-hypergeometric orthogonal
polynomials, \emph{Trans. Amer. Math. Soc.} \textbf{350} (8) (1998), 3269--3296.
\bibitem[Doi92]{Doi92} Y. Doi, Unifying Hopf modules, \emph{J. Algebra} \textbf{153} (1992), 373--385. 
\bibitem[Dri87]{Dri87} V. G. Drinfeld, Quantum groups, in: \emph{Proceedings of the International Congress of Mathematicians}, \textbf{1}, \textbf{2} (Berkeley, Calif., 1986), \emph{Amer. Math. Soc.} (1987), 798--820.
\bibitem[Dri93]{Dri93} V. G. Drinfeld, On Poisson homogeneous spaces of Poisson-Lie groups, \emph{Teoret. Mat. Fiz.} \textbf{95} (2) (1993), 226--227.
\bibitem[ES98]{ES98} P. Etingof, O. Schiffmann, Lectures on Quantum Groups, \emph{Lectures in Math. Phys.}, International Press (1998).
\bibitem[Jim85]{Jim85} M. Jimbo, A $q$-difference analog of $U(\mfg)$ and the Yang-Baxter equation, \emph{Lett. Math. Phys.} \textbf{10} (1985), 63--69.
\bibitem[KK03]{KK03} E. Koelink and J. Kustermans, A locally compact quantum group analogue of the normalizer of $SU(1,1)$ in $SL(2,\C)$, \emph{Comm. Math. Phys.} \textbf{233} (2003), 231-296.
\bibitem[Kol14]{Kol14} S. Kolb, Quantum symmetric Kac-Moody pairs, \emph{Adv. Math.} \textbf{267} (2014), 395-469.\bibitem[Koo93]{Koo93} T. Koornwinder, Askey-Wilson polynomials as zonal spherical functions on the $SU(2)$ quantum group, \emph{SIAM Journal Math. Anal.} \textbf{24} (N3) (1993), 795--813.
\bibitem[Kop95]{Kop95} M. Koppinen, Variations on the smash product with applications to group-graded rings, \emph{J. Pure Appl. Algebra} \textbf{104} (1995), 61--80. 
\bibitem[Kor94]{Kor94} L.I. Korogodsky, Quantum Group $SU(1, 1) \rtimes \mathbb{Z}_2$ and super tensor products, \emph{Commun. Math. Phys.} \textbf{163} (1994), 433--460.
\bibitem[Let99]{Let99} G. Letzter, Symmetric pairs for quantized enveloping algebras, \emph{J. Algebra} \textbf{220} (2) (1999), 729-767.
\bibitem[Let02]{Let02} G. Letzter, Coideal subalgebras and quantum symmetric pairs, \emph{New directions in Hopf algebras, Math. Sci. Res. Inst. Publ.} \textbf{43}, Cambridge Univ. Press, Cambridge (2002), 117--165.
\bibitem[LS91]{LS91} S. Levendorskii and Y. Soibelman, Algebras of Functions on Compact Quantum Groups, Schubert Cells and Quantum Tori, \emph{Commun. Math. Phys.} \textbf{139} (1991), 141--170.
\bibitem[MS99]{MS99} E. F. M\"{u}ller and H.-J. Schneider, Quantum homogeneous spaces with faithfully flat module structures, \emph{Israel J. Math.} \textbf{111} (1999), 157--190. 
\bibitem[NS95]{NS95} M. Noumi and T. Sugitani, Quantum symmetric spaces and related q-orthogonal polynomials, In: \emph{Group Theoretical Methods in Physics (ICGTMP) (Toyonaka, Japan, 1994), World Sci. Publishing, River Edge, N.J.} (1995), 28--40.
\bibitem[Nou96]{Nou96} M. Noumi, Macdonalds symmetric polynomials as zonal spherical functions on some quantum homogeneous spaces, \emph{Adv. Math.} \textbf{123} (1) (1996), 16--77.
\bibitem[Sch00]{Sch00} P. Schauenburg, Doi-Koppinen Hopf Modules Versus Entwined Modules, \emph{New York J. Math.} \textbf{6} (2000), 325--329.
\bibitem[Tak79]{Tak79} M. Takeuchi, Relative Hopfmodules - Equivalences and freeness criteria, \emph{J. Algebra} \textbf{60} (1979), 452--471.
\bibitem[Tak80]{Tak80} M. Takeuchi, $\mathrm{Ext}_{\mathrm{ad}}(SpR,\mu^A) \cong \widehat{\mathrm{Br}}(A/k)$, \emph{J. Algebra} \textbf{67} (1980), 436--475.
\bibitem[Twi92]{Twi92} E. Twietmeyer, Real forms of $U_q(\mathfrak{g})$, \emph{Lett. Math. Phys.} \textbf{24} (1992), 49-58.
\bibitem[VY20]{VY20} C. Voigt and R. Yuncken, Complex Semisimple Quantum Groups and Representation Theory, \emph{Lecture Notes in Mathematics} \textbf{2264}, Springer International Publishing, 376+X pages.  
\bibitem[VY23]{VY23} C. Voigt and R. Yuncken, The Plancherel formula for complex semisimple quantum groups, \emph{Ann. Sci. \'{E}c. Norm . Sup\'{e}r.} \textbf{56} (1), 299--322.
\bibitem[Wor87]{Wor87} S. L. Woronowicz, Compact matrix pseudogroups, \emph{Comm. Math. Phys.} \textbf{111} (1987), 613--665.
\bibitem[Wor88]{Wor88} S. L. Woronowicz, Tannaka-Krein duality for compact matrix pseudogroups. Twisted $SU(N)$ groups, \emph{Invent. Math.} \textbf{93} (1) (19), 35--76.
\bibitem[Wor91]{Wor91} S.L. Woronowicz, Unbounded elements affiliated with C$^*$-algebras and noncompact quantum groups, \emph{Comm. Math. Phys.} \textbf{136} (1991), 399--432.
\bibitem[Wor00]{Wor00} S.L. Woronowicz, Extended $SU(1,1)$ quantum group. Hilbert space level, Preprint KMMF (unfinished) (2000).
\end{thebibliography}
\end{document}